\documentclass{article}
\usepackage{fancyhdr}
\usepackage{amscd}
\usepackage[dvipdfmx]{graphicx}
\usepackage{amsmath}
\usepackage{amssymb}
\usepackage{amsthm}
\usepackage{mathrsfs}
\newtheorem{dfn}{Definition}[section]
\newtheorem{thm}[dfn]{Theorem}
\newtheorem{prop}[dfn]{Proposition}
\newtheorem{lem}[dfn]{Lemma}
\newtheorem{cor}[dfn]{Corollary}
\newtheorem{rem}[dfn]{Remark}

\newtheorem{ass}[dfn]{Assumption}

\numberwithin{equation}{section}

\usepackage{color}
\begin{document}

\title{Blow-up behavior for ODEs with normally hyperbolic nature in dynamics at infinity}

\author{Kaname Matsue\thanks{Institute of Mathematics for Industry, Kyushu University, Fukuoka 819-0395, Japan {\tt kmatsue@imi.kyushu-u.ac.jp}} $^{,}$ \footnote{International Institute for Carbon-Neutral Energy Research (WPI-I$^2$CNER), Kyushu University, Fukuoka 819-0395, Japan}
}
\maketitle

\begin{abstract}
We describe blow-up behavior for ODEs by means of dynamics at infinity with complex asymptotic behavior in autonomous systems, as well as in nonautonomous systems.
Based on preceding studies, a variant of closed embeddings of phase spaces and the time-scale transformation determined by the structure of vector fields at infinity reduce our description of blow-ups to unravel the shadowing property of (pre)compact trajectories on the horizon, the geometric object expressing the infinity, with the specific convergence rates.
Geometrically, this description is organized by asymptotic phase of invariant sets on the horizon.
Blow-up solutions in nonautonomous systems can be described in a similar way.
As a corollary, normally, or partially hyperbolic invariant manifolds on the horizon possessing asymptotic phase are shown to induce blow-ups.
\end{abstract}


\section{Introduction}

Blow-up behavior of solutions in differential equations is widely studied in decades from many aspects such as mathematical, numerical and physical ones.
Fundamental questions in blow-up phenomena are {\em whether or not a solution blows up at a finite time} and, if it does, {\em when, where} and {\em how} it blows up.
While numerous studies involving blow-ups and related {\em finite-time singularities} such as finite-time extinction, collapse and  quenching are provided from various viewpoints, the author and his collaborators have provided characterizations of finite-time singularities involving blow-ups by means of {\em dynamics at infinity} based on compactifications of phase spaces and appropriate time-scale desingularizations \cite{Mat2018, Mat2019, asym1, asym2}.
These machineries reduce the problem for finding and characterizing blow-up solutions to that for characterizing (center-)stable manifolds of equilibria/periodic orbits on {\em the horizon}, the geometric object expressing infinity as the boundaries of compactified manifolds or their upper-tangent spaces, for the transformed vector fields called {\em desingularized vector fields} so that \lq\lq dynamics at infinity" makes sense.
In particular, asymptotic theory in dynamical systems provides various characterizations of blow-up solutions under appropriate setting and transform in vector fields.
Note that this approach is applied to studying features of solutions of PDEs with blow-up and/or quenching behavior (e.g., \cite{IMS2020, IS2020, IS2021, IS2022}), while the approach itself is also used to study bounded objects in dynamical systems (e.g., \cite{DH1999, DLA2006, GKO2022, I2023_2, KR2004}).
Moreover, the proposed machinery also provides various achievements in {\em computer-assisted proofs} of blow-up solutions with well-established concepts in dynamical systems \cite{LMT2023, MT2020_1, MT2020_2, TMSTMO2017}.
\par
In the previous studies, blow-up solutions are considered only for {\em autonomous} vector fields ${\bf y}' = f({\bf y})$ admitting, as already mentioned, {\em equilibria} and/or {\em periodic orbits} at infinity.
While various generalization will be proposed, we shall pay our attention to the following generalizations here:
\begin{itemize}
\item blow-up behavior with more complex nature at infinity, and
\item characterization of blow-ups in nonautonomous systems. 
\end{itemize}
One of appropriate candidates in the first direction would be blow-ups shadowing {\em normally hyperbolic invariant manifolds} ({\em NHIMs} for short) at infinity (cf. \cite{K2014} for characterizing unbounded geometry by means of NHIMs){\color{black}, where several fundamental properties characterizing blow-ups associated with hyperbolic equilibria and periodic orbits are admitted}.
In the case of nonautonomous systems
\begin{equation}
\label{ODE-original}
{\bf y}' = f(t, {\bf y})
\end{equation}
under a suitable setting to $f$, an appropriate treatment of time variable $t$ is required. 
One of typical treatments is to regard the time variable $t$ as an additional phase variable and the original system as the extended autonomous system
\begin{equation}
\label{nonaut-extend}
\frac{d{\bf y}}{d\eta} = f(t,{\bf y}),\quad \frac{dt}{d\eta} = 1,
\end{equation}
in which case treatments of invariant sets becomes different from general autonomous systems\footnote{
For example, {\em equilibria} do not make sense in (\ref{nonaut-extend}).
}.
While useful machineries such as {\em compactifications in the time variable} (e.g., \cite{WXJ2021}) are established for studying global-in-time solutions in nonautonomous systems, we will see that the above extended autonomous system (\ref{nonaut-extend}) with an appropriate rule of {\em scaling} in $t$ is suitable for characterizing blow-ups in nonautonomous systems in the sense that machineries developed for autonomous systems can be applied.
\par
Our main aims here are to see that the following stuff essentially characterize blow-ups both in autonomous and nonautonomous systems:
\begin{itemize}
\item shadowing property of trajectories at infinity;
\item \lq\lq specific (and sufficiently fast) convergence rate" to them.
\end{itemize}
Geometrically, the first requirement is described by means of {\em asymptotic phase} of invariant sets (e.g., \cite{LP2021} and references therein), which is manifested in invariant manifolds with {\em partial / normal hyperbolicity}.
An example is invariant foliations of stable manifolds of NHIMs.
The second requirement is {\color{black}a little} stronger than usual characterizations of hyperbolicity.
In a special case, we refer to the theory in {\em linear differential systems}, e.g., \cite{SS1974, SS1978, SS1980} (see also \cite{KR2011, S1971}), where {\em dichotomy spectrum} characterize stability or asymptotic behavior of solutions.
We shall see that a special structure of the spectrum provides our requirement of \lq\lq convergence rate", which includes the case of the presence of hyperbolic equilibria/periodic orbits at infinity (\cite{Mat2018}).
\par
\bigskip
The rest of the present paper is organized as follows.
In Section \ref{section-preliminaries-NH}, we review a fundamental concept of NHIMs and put a remark on asymptotic phase.
In Section \ref{section-preliminaries}, we summarize fundamental settings of vector fields we shall treat throughout the present paper and machineries to characterize dynamics at infinity including blow-up solutions; embeddings of phase spaces and time-scale desingularizations.
In Section \ref{section-blowup-NHIM}, we pay our attention to autonomous systems, and provide a {\color{black}description} of blow-up solutions {\em shadowing trajectories at infinity}.
This characterization generalizes preceding works where blow-up solutions shadowing hyperbolic equilibria and periodic orbits at infinity are characterized (e.g., \cite{Mat2018, Mat2019}).
We further describe blow-up solutions with asymptotic rates automatically determined by the quasi-homogeneity of vector fields, which are often referred to as {\em type-I} blow-ups, under the specific \lq\lq convergence rate" to shadowed trajectories.
In Section \ref{section-blowup-nonaut}, we move to nonautonomous systems and derive a criterion of the existence of blow-ups with their asymptotic behavior.
In addition to a fundamental description, a geometric treatment of NHIMs with boundaries towards our aims is also presented to geometrically describe the blow-up criterion.
In Section \ref{section-blowup-ex}, we show several examples of blow-up solutions in nonautonomous systems.
Our examples shown there revisit to preceding works in various backgrounds to characterize finite-time singularities, and will be shown that blow-ups are described {\color{black}{\em in a unified way} through the present methodology}.
\par
Several technical details are collected in Appendix.
In Appendix \ref{sec:dir}, another embedding frequently applied in practical problems is reviewed. 
In Appendix \ref{section-SSspec}, a sufficient condition to satisfy the rate requirement discussed in Section \ref{section-blowup-NHIM} is provided by means of {\em spectral theory in linear differential systems}.
In Appendix \ref{section-technical_nonaut}, a precise description of blow-up rates of blow-up solutions in nonautonomous systems via a geometric technique is presented.

\section{Preliminaries 1: normally hyperbolic invariant manifolds}
\label{section-preliminaries-NH}

Here we briefly summarize a fundamental concept in dynamical systems we shall use later, {\em normally hyperbolic invariant manifolds} (NHIMs) followed by \cite{EKR2018, HPS1977}, as well as a remark on a property which NHIMs admit, {\em asymptotic phase}.
\begin{dfn}[Vector bundles. e.g., \cite{H1994, MS1974, SS1974}]\rm
An $n$-dimensional {\em vector bundle} $\xi = (\pi, E, X)$ (over $X$) is defined as the triple of the following objects: a topological space $X$ called the {\em base space}, the {\em total space} $E = \bigcup_{{\bf x}\in X}F_{\bf x}$ with the {\em fiber} $F = F_{\bf x} \cong \mathbb{R}^n$, and the {\em continuous projection} $\pi: E\to X$.
For any subset $M\subset X$ of the base space, let $\xi|_M = (\pi, E(M), M)$ with $E(M) = \bigcup_{{\bf x}\in M}F_{\bf x}$ be the restriction of $\xi$ over $M$.
\end{dfn}

Let $f: \mathbb{R}^n\to \mathbb{R}^n$ be a $C^r$-vector field with $r \geq 1$ and $\varphi = \varphi_f$ be the generated flow.
For any $t\in \mathbb{R}$, the associated time-$t$ map is denoted by $\varphi^t = \varphi_f^t$. 
For any subset $M\subset \mathbb{R}^n$, let $T_M\mathbb{R}^n := \pi^{-1}(M)$ associated with the vector bundle $(\pi, T\mathbb{R}^n, \mathbb{R}^n)$.

\begin{rem}
Arguments in this subsection are valid for $C^r$-vector fields $f: Q\to TQ$, where $Q$ is a $C^r$ Riemannian manifold.
Nevertheless we pay our attention to the case $Q=\mathbb{R}^n$ for simplicity.
Indeed, our interests here are submanifolds of $\mathbb{R}^n$ for some $n$ and associated vector fields.
\end{rem}

\begin{dfn}[Normal hyperbolicity, cf. \cite{HPS1977, EKR2018}]\rm
\label{dfn-NH}
Let $M\subset \mathbb{R}^n$ be a compact, connected manifold with $\partial M = \emptyset$ which is invariant for $\varphi$.
We say that $M$ is {\em $r$-normally hyperbolic}\footnote{
In \cite{HPS1977}, the above definition is referred to as {\em $r$-eventually relatively normal hyperbolicity}.
This property is sufficient in our purpose.
} if there is a continuous splitting 
\begin{equation}
\label{cont-splitting}
T_M \mathbb{R}^n = TM\oplus {\color{black}E^s \oplus E^u}
\end{equation}
of $T_M\mathbb{R}^n$ such that, for any $t$, all subbundles $TM$, $E^u$ and $E^s$ are $D\varphi^t$-invariant and that there are constants $C>0$ and $\lambda_s < 0 < \lambda_u$ satisfying
\begin{equation}
\label{exp-NH}
\|D\varphi^t|_{E^s_p}\| \leq Ce^{\lambda_s t} m\left( D\varphi^t|_{T_pM} \right)^i, \quad m\left( D\varphi^t|_{E^u_p} \right) \geq \frac{e^{\lambda_u t}}{C} \|D\varphi^t|_{T_pM}\|^i
\end{equation}
for all $p\in M$, $t\geq 0$, and $0\leq i \leq r$.
We shall call $M$ an ($r$-){\em normally hyperbolic invariant manifold}, or ($r$-){\em NHIM} for short.
Now $\oplus$ in (\ref{cont-splitting}) denotes the Whitney sum of vector bundles and 
\begin{equation*}
m(A) := \inf \left\{\|Av\| \mid \|v\| = 1 \right\}
\end{equation*}
is the minimum norm\footnote{
If $A$ is invertible, then $m(A) = \|A^{-1}\|^{-1}$.
} of a matrix $A$.
We particularly say a NHIM $M$ being {\em normally attracting} ($M$ being a {\em NAIM} for short) if $E^u$ is the zero bundle; $E^u = \{0\}$.
\end{dfn}

Denote $n_T, {\color{black}n_s, n_u}$ with $n = n_T + n_s+ n_u$ by the rank ($=$ dimension of fibers) of $TM$, {\color{black}$E^s$ and $E^u$}, respectively.
In particular, $TM$, {\color{black}$E^s$ and $E^u$} are vector bundles with the base space $M$ and $n_T$-, {\color{black}$n_s$- and $n_u$-}dimensional fibers.
We shall call {\color{black}$E^s$ and $E^u$ {\em the stable bundle} and {\em the unstable bundle} of $M$}, respectively, while $TM$ is the tangent bundle of $M$.
When $M$ is a NAIM, then $n_u = 0$.
The vector bundle structure of $E^s$ is written by $\xi^s = (\pi_s, E^s, F^s)$
with 
\begin{equation*}
E^s = \bigcup_{p\in M} F^s_p,\quad F^s \cong F^s_p\cong \mathbb{R}^{n_s}
\end{equation*}
for all $p\in M$. 
A similar description holds for $E^u$.
Several properties which NHIMs possess are summarized below.
\begin{prop}[e.g., \cite{HPS1977, PT1977, PS1970, W2013}]
\label{prop-summary-NHIM}
Let $M\subset \mathbb{R}^n$ be a boundaryless, compact NHIM for $\varphi_f$.
\begin{enumerate}
\item (Invariant foliation). There is a {\em local stable manifold} $W^s_{\rm loc}(M)$ tangent to the Whitney sum $TM\oplus E^s$ at $M$, which is a $C^r$-embedded submanifold of $\mathbb{R}^n$ and invariantly fibered by embedded ($n_s$-dimensional) disks $D_p^s$, denoted by $W^s_{\rm loc}(p)$ for each $p\in M$, constructing fibers of the {\em local stable invariant foliation}
\begin{equation*}
W^s_{\rm loc}(M) (=  W^s_{\rm loc}(M; \varphi)) = \bigcup_{p\in M}W^s_{\rm loc}(p).
\end{equation*}
In particular, $W^s_{\rm loc}(p)$ is an invariant manifold satisfying $T_pW^s_{\rm loc}(p) = E_p^s$ for each $p\in M$.
\item (Exponential decay along fibers). There exist $C_s > 0$ and $\lambda_s < 0$ such that, for each $p\in M$, the following holds for all $t\geq 0$ whenever $q\in W^s_{\rm loc}(p)$:
\begin{equation}
\label{estimate-NHIM-exp}
\|\varphi_f(t, q) - \varphi_f(t, p)\| < C_s e^{\lambda_s t}.
\end{equation}
\end{enumerate}
\end{prop}

\par
\bigskip
Another classical property; {\em local conjugacy of dynamics to linearizations}, is left to Appendix \ref{section-conjugacy}.
The theory of NHIMs is also established for manifolds {\em with nonempty boundary}.
We collect basic notions of normal hyperbolicity in such a case.
\begin{dfn}[Inflowing/overflowing invariant manifolds, e.g., \cite{W2013}]\rm
\label{dfn_in_overflow}
Suppose that $M$ is a compact manifold in an ambient manifold $Q$ possibly with nonempty boundary: $\partial M \not = \emptyset$.
We say that $M$ is {\em inflowing invariant} for the vector field $f$ if $\varphi_f^t(M)\subset M$ for all $t\geq 0$ and if $f$ points strictly inward at $\partial M$.
Similarly, we say that $M$ is {\em overflowing invariant} for the vector field $f$ if $\varphi_f^t(M)\subset M$ for all $t\leq 0$ and if $f$ points strictly outward at $\partial M$. 
If $M$ is inflowing/overflowing invariant with the structures (\ref{cont-splitting}) and (\ref{exp-NH}), we call $M$ {\em inflowing/overflowing NHIM}, respectively\footnote{
If $\partial M = \emptyset$ and $M$ is invariant, then $M$ is both overflowing and inflowing invariant.
But if $\partial M \not = \emptyset$, $M$ being invariant does {\em not} imply that $M$ is both overflowing and inflowing invariant.
}.
\end{dfn}
It is proved that{\color{black}, under suitable contracting/expanding properties like normal hyperbolicity}, {\em inflowing} invariant manifolds admit invariant foliations of local {\em stable} manifolds, while {\em overflowing} invariant manifolds admit invariant foliations of local {\em unstable} manifolds (e.g., \cite{EKR2018, W2013}).

Finally we shall make a remark on invariant foliations; shadowing property for NHIMs.
Even if an invariant manifold is not normally hyperbolic, there are several cases which it admits the similar property.
\begin{rem}[Asymptotic phase, cf. \cite{LP2018, LP2021}]
\label{rem-asym-phase}
Let $X = (X,d)$ be a metric space and $\varphi$ be a flow on $X$.
Suppose that there is an attracting invariant set $A$ with a basin $B$: 
\begin{equation*}
\lim_{t\to \infty} d(\varphi^t({\bf x}), A) = 0,\quad \forall {\bf x}\in B.
\end{equation*}
It is said that a mapping $t\mapsto \varphi^t({\bf x})$, ${\bf x}\in B$, has an {\em asymptotic phase} if there is ${\bf x}_\ast \in A$ such that
\begin{equation*}
\lim_{t\to \infty} d(\varphi^t({\bf x}), \varphi^t({\bf x}_\ast)) = 0.
\end{equation*}
Stable invariant foliation for a NHIM $M$ is a special case of asymptotic phase: $A=M$ and $B=W^s_{\rm loc}(M)$.
Asymptotic phase beyond normal hyperbolicity is constructed in several studies.
See e.g., \cite{LP2018, LP2021} for a brief summary of historical backgrounds in asymptotic phase and its existence in a generalized setting.
\end{rem}

\section{Preliminaries 2: dynamics at infinity}
\label{section-preliminaries}
In this section we provide tools and methodologies to consider divergent solutions in terms of global-in-time trajectories converging to invariant sets for appropriately transformed vector fields as general as possible.

\subsection{Asymptotically quasi-homogeneous systems}
\label{section-AQH}
First of all, we collect our present settings precisely, which are based on \cite{Mat2018}.

\begin{dfn}[Homogeneity index and admissible domain]\rm
\label{dfn-index}
Let $\alpha = (\alpha_1,\cdots, \alpha_n)$ be a set of nonnegative integers.
We say the index set $I_\alpha=\{i\in \{1,\cdots, n\}\mid \alpha_i > 0\}$ the set of {\em homogeneity indices associated with $\alpha = (\alpha_1,\cdots, \alpha_n)$}.
Let $U\subset \mathbb{R}^n$.
We say the domain $U\subset \mathbb{R}^n$ {\em admissible with respect to the sequence $\alpha$}
if
\begin{equation*}
U = \left\{{\bf x}=(x_1,\cdots, x_n)\in \mathbb{R}^n \mid x_i\in \mathbb{R}\,\text{ if }\, i\in I_\alpha,\ (x_{j_1},\cdots, x_{j_{n-l}}) \in \tilde U\right\},
\end{equation*} 
where $\{j_1, \cdots, j_{n-l}\} = \{1,\cdots, n\}\setminus I_\alpha$ and $\tilde U$ is an open set in $\mathbb{R}^{n-l}$ spanning $(x_{j_1},\cdots, x_{j_{n-l}})$ with $\{j_1, \cdots, j_{n-l}\} = \{1,\cdots, n\}\setminus I_\alpha$.
\end{dfn}


\begin{dfn}[Asymptotically quasi-homogeneous vector fields, cf. \cite{D1993, Mat2018}]\rm
\label{dfn-AQH}
Let $U$ be an admissible set $U\subset \mathbb{R}^n$ with respect to $\alpha$. 
Also, let $f_0:U \to \mathbb{R}$ be a function.
Let $\alpha_1,\ldots, \alpha_n$ be nonnegative integers with $(\alpha_1,\ldots, \alpha_n) \not = (0,\ldots, 0)$ and $k > 0$.
We say that $f_0$ is a {\em quasi-homogeneous function\footnote{
In preceding studies, all $\alpha_i$'s and $k$ are typically assumed to be natural numbers.
In the present study, on the other hand, the above generalization is valid.
} of type $\alpha = (\alpha_1,\ldots, \alpha_n)$ and order $k$} if
\begin{equation*}
f_0( s^{\Lambda_\alpha}{\bf x} ) = s^k f_0( {\bf x} )\quad \text{ for all } {\bf x} = (x_1,\ldots, x_n)^T \in U \text{ and } s>0,
\end{equation*}
where\footnote{
Throughout the rest of this paper, the power of real positive numbers or functions to matrices is described in the similar manner.
}
\begin{equation*}
\Lambda_\alpha =  {\rm diag}\left(\alpha_1,\ldots, \alpha_n\right),\quad s^{\Lambda_\alpha}{\bf x} = (s^{\alpha_1}x_1,\ldots, s^{\alpha_n}x_n)^T.
\end{equation*}
Next, let $X = \sum_{i=1}^n f_i({\bf x})\frac{\partial }{\partial x_i}$ be a continuous vector field defined on $U$.
We say that $X$, or simply $f = (f_1,\ldots, f_n)^T$ is a {\em quasi-homogeneous vector field of type $\alpha = (\alpha_1,\ldots, \alpha_n)$ and order $k+1$} if each component $f_i$ is a quasi-homogeneous function of type $\alpha$ and order $k + \alpha_i$.
\par
Finally, we say that $X = \sum_{i=1}^n f_i({\bf x})\frac{\partial }{\partial x_i}$, or simply $f: U\to \mathbb{R}^n$ is an {\em asymptotically quasi-homogeneous vector field of type $\alpha = (\alpha_1,\ldots, \alpha_n)$ and order $k+1$ (at infinity)} if there is a quasi-homogeneous vector field  $f_{\alpha,k} = (f_{i; \alpha,k})_{i=1}^n$ of type $\alpha$ and order $k+1$ such that
\begin{equation}
\label{residual}
f_i( s^{\Lambda_\alpha}{\bf x} ) - s^{k+\alpha_i} f_{i;\alpha,k}( {\bf x} ) = o(s^{k+\alpha_i}),\quad i\in \{1,\ldots, n\},
 \end{equation}
as $s\to +\infty$ uniformly on $\left\{{\bf x}\in U \mid \sum_{i\in I_\alpha} x_i^2 = 1, (x_{j_1},\cdots, x_{j_{n-l}}) \in \tilde K\right\}$ for any compact subset $K\subset \tilde U$.
\end{dfn}

\begin{rem}
In the above definition, non-polynomial-like functions such as $\sin x$ are not included to characterize quasi-homogeneity.
Indeed, such functions are allowed to exist only in the residual terms characterized by the asymptotic quasi-homogeneity (\ref{residual}). 
On the other hand, (\ref{residual}) is required for all $i\in \{1,\ldots, n\}$.
\end{rem}

A fundamental property of quasi-homogeneous functions and vector fields is reviewed in e.g. \cite{asym1}.
%
%
%
%
%
Throughout the rest of this section, consider an (autonomous) $C^r$ vector field (\ref{ODE-original}) with $r\geq 1$, where $f: U \to \mathbb{R}^n$ is asymptotically quasi-homogeneous of type $\alpha = (\alpha_1,\ldots, \alpha_n)$ and order $k+1$ at infinity defined on an admissible set $U\subset \mathbb{R}^n$ with respect to $\alpha$.

%
%
%
%

\subsection{Embedding of phase spaces and dynamics at infinity}

Our process to describe blow-up solutions from the viewpoint of dynamical systems is summarized below.
\begin{enumerate}
\item For given (asymptotically quasi-homogeneous) vector field $f$, determine its type $\alpha$ and order $k+1$.
See Section \ref{section-blowup-nonaut} for the treatment of $t$ in nonautonomous systems.
\item Apply an {\em embedding of phase spaces} with the type $\alpha$, precisely defined below, and transform $f$ through this embedding.
\item Introduce the time-scale transformation determined by the order $k+1$ of $f$ to desingularize the transformed vector field at infinity.
\end{enumerate}
The above process is widely used in {\em autonomous} systems (cf. \cite{I2023_1, I2023_2, IMS2020, IS2020, IS2021, IS2022, LMT2023, Mat2018, Mat2019, MT2020_1, TMSTMO2017}) and we follow this strategy in the present argument as well as nonautonomous systems later.

\begin{rem}
In preceding studies, embeddings we would review below were referred to as {\em compactifications}.
On the other hand, because the embedded manifolds involving our interests here are typically non-compact, including the case of nonautonomous systems, we shall use the terminology \lq\lq embeddings" instead of \lq\lq compactifications".
\end{rem}

In the following arguments in this section, we fix a type $\alpha = (\alpha_1,\ldots, \alpha_n)\in \mathbb{Z}_{\geq 0}^n \setminus \{(0,\ldots, 0)\}$, and let $I_\alpha$ be the associated set of homogeneity indices.
Also, let $U\subset \mathbb{R}^n$ be an admissible set with respect to $\alpha$.
A \lq\lq global-type" embedding and the associated vector field suitable for describing dynamics at infinity are introduced here.
Another example of embeddings is shown in Appendix \ref{sec:dir}.

\begin{dfn}[Quasi-parabolic embedding, cf. \cite{MT2020_1}]\rm
\label{dfn-quasi-para}
Let $\{\beta_i\}_{i\in I_\alpha}$ be the collection of natural numbers so that 
\begin{equation}
\label{LCM}
\alpha_i \beta_i \equiv c \in \mathbb{N},\quad i\in I_\alpha
\end{equation}
is the least common multiplier.
In particular, $\{\beta_i\}_{i\in I_\alpha}$ is chosen to be the smallest among possible collections.
Let $p({\bf y})$ be a functional given by
\begin{equation}
\label{func-p}
p({\bf y}) \equiv \left( \sum_{i \in I_\alpha} y_i^{2\beta_i} \right)^{1/2c}.
\end{equation}
Define the mapping $T_{{\rm para};\alpha}: U \to \mathbb{R}^n$ as the inverse of
\begin{equation}
\label{parabolic-cpt}
S_{{\rm para};\alpha}({\bf x}) = {\bf y},\quad y_j = \kappa^{\alpha_j} x_j,\quad j=1,\ldots, n,
\end{equation}
where 
\begin{equation*}
\kappa = \kappa({\bf x}) = (1- p({\bf x})^{2c})^{-1} \equiv \left( 1 - \sum_{j\in I_\alpha} x_j^{2\beta_j}\right)^{-1}.
\end{equation*}
We say the mapping $T_{{\rm para};\alpha}$ the {\em quasi-parabolic embedding (with type $\alpha$)}.
\end{dfn}

\begin{rem}
\label{rem-kappa}
The functional $\kappa = \tilde \kappa({\bf y})$ as a functional determined by ${\bf y}$ is implicitly determined by $p({\bf y})$.
Details of such a characterization of $\kappa$ in terms of ${\bf y}$, and the bijectivity and smoothness of $T_{{\rm para};\alpha}$ are shown in \cite{MT2020_1} with a general class of embeddings including quasi-parabolic compactifications (embeddings).
\par
The above definition is originally introduced in the case $I_\alpha = \{1,\ldots, n\}$.
It is nevertheless valid for the general case, $I_\alpha \not = \{1,\ldots, n\}$, to characterize infinity suitable for the behavior of $f$.
It should be however noted that $T_{{\rm para};\alpha}$ and its inverse $S_{{\rm para};\alpha}$ are defined for {\em all} $j \in \{1,\ldots, n\}$, not only $j\in I_\alpha$.
\end{rem}

As proved in \cite{MT2020_1}, $T_{{\rm para};\alpha}$ maps $U$ one-to-one onto the set
$\mathcal{D} \equiv \{{\bf x}\in U \mid p({\bf x}) < 1\}$.
Infinity in the original coordinate then corresponds to a point on the level set of $p$:
\begin{equation*}
\mathcal{E} = \{{\bf x} \in U \mid p({\bf x}) = 1\}.
\end{equation*}
\begin{dfn}\rm
We call the set $\mathcal{E}$ the {\em horizon}.
\end{dfn}

Once we fix an embedding associated with the type $\alpha = (\alpha_1, \ldots, \alpha_n)$ of the vector field $f$ with order $k+1$, we can derive the vector field which {\color{black}is defined} including the horizon.
Then the {\em dynamics at infinity} makes sense through the appropriately transformed vector field called the {\em desingularized vector field}, denoted by $g$.
The common approach is twofold.
Firstly, we rewrite the vector field (\ref{ODE-original}) with respect to the new variable defined by embeddings.
Secondly, we introduce the time-scale transformation of the form $d\tau = q({\bf x})\kappa({\bf x}(t))^k dt$ for some function $q({\bf x})$ which is bounded including the horizon. 
We then obtain the vector field with respect to the new time variable $\tau$, which is continuous, including the horizon.

\begin{rem}
\label{rem-choice-cpt}
Continuity of the desingularized vector field $g$ including the horizon is guaranteed by the smoothness of $f$ and asymptotic quasi-homogeneity (\cite{Mat2018}).
In the case of parabolic-type embeddings introduced here, $g$ inherits the smoothness of $f$ including the horizon, which is not always the case of other embeddings in general. 
Details are discussed in \cite{Mat2018}.
\end{rem}

\begin{dfn}[Time-scale desingularization]\rm 
Define the new time variable $\tau$ by
\begin{equation}
\label{time-desing-para}
d\tau = q({\bf x})^{-1}(1-p({\bf x})^{2c})^{-k} dt,
\quad q({\bf x}) = 1-\frac{2c-1}{2c}(1-p({\bf x})^{2c}),
\end{equation}
equivalently
\begin{equation*}
t - t_0 = \int_{{\color{black}0}}^\tau q({\bf x}(\tau)) (1-p({\bf x}(\tau))^{2c})^k d\tau,
\end{equation*}
where {\color{black}$t_0$ denotes the initial time in the $t$-timescale}, ${\bf x}(\tau) = T_{{\rm para};\alpha}({\bf y}(\tau))$ and ${\bf y}(\tau)$ is a solution ${\bf y}(t)$ under the parameter $\tau$.
We shall call (\ref{time-desing-para}) {\em the time-scale desingularization of order $k+1$}.
\end{dfn}
The change of coordinate and the above desingularization yield the following vector field $g = (g_1, \ldots, g_n)^T$, which is continuous on $\widetilde{\mathcal{D}} \equiv \mathcal{D}\cup \mathcal{E}$:
\begin{align*}
\dot x_i \equiv \frac{dx_i}{d\tau} = g_i({\bf x}) = q({\bf x}) \left\{ \tilde f_i({\bf x}) - \alpha_i x_i \sum_{j=1}^n (\nabla \kappa)_j \kappa^{\alpha_j - 1}\tilde f_j({\bf x})\right\},
\end{align*}
where 
\begin{equation}
\label{f-tilde}
\tilde f_j({\bf x}) := \kappa^{-(k+\alpha_j)} f_j(\kappa^{\Lambda_\alpha}{\bf x}),\quad j=1,\ldots, n,
\end{equation}
and
$\nabla \kappa = \nabla_{\bf x} \kappa = ((\nabla_{\bf x} \kappa)_1, \ldots, (\nabla_{\bf x} \kappa)_n)^T$ is
\begin{equation*}
(\nabla_{\bf x} \kappa)_j 
= \begin{cases}
\frac{\kappa^{1-\alpha_j} x_j^{2\beta_j-1}}{\alpha_j q({\bf x}) } & j\in I_\alpha, \\
0 & \text{otherwise},
\end{cases}
\end{equation*}
as derived in \cite{MT2020_1}.
In particular, the vector field $g$ and the differential equation is simply written as follows:
\begin{align}
\label{desing-para}
\dot {\bf x} = g({\bf x}) = q({\bf x}) \tilde f({\bf x}) - G({\bf x})\Lambda_\alpha {\bf x},
\end{align}
where $\tilde f = (\tilde f_1,\ldots, \tilde f_n)^T$ and
\begin{align}
\label{Gx}
G({\bf x}) &\equiv \sum_{j\in I_\alpha} \frac{x_j^{2\beta_j-1}}{\alpha_j}\tilde f_j({\bf x}).
\end{align}
Smoothness of $f$ and its asymptotic quasi-homogeneity guarantee the smoothness of the right-hand side $g$ of (\ref{desing-para}) including the horizon $\mathcal{E}\equiv \{p({\bf x}) = 1\}$.
In particular, {\em dynamics at infinity}, such as divergence of solutions to specific directions,  is characterized through dynamics generated by (\ref{desing-para}) around the horizon. 
See \cite{Mat2018, MT2020_1} for details.

\begin{rem}[Invariant structure]
\label{rem-invariance}
The horizon $\mathcal{E}$ is invariant in $\widetilde{\mathcal{D}}$. 
Indeed, direct calculations yield that
\begin{equation*}
\left. \frac{d}{d \tau}p({\bf x}(\tau))^{2c}\right|_{\tau=0} = 0\quad \text{ whenever }\quad {\bf x}(0)\in \mathcal{E}.
\end{equation*}
In particular, $\mathcal{E}$ is a codimension-one invariant submanifold of $\widetilde{\mathcal{D}}$ if $I_\alpha = \{1,\ldots, n\}$.
See e.g. \cite{Mat2018}, where detailed calculations are shown in a similar type of global embeddings.
We shall apply this invariant structure to extracting the detailed blow-up structure later.
\end{rem}

\section{Normally hyperbolic blow-up}
\label{section-blowup-NHIM}

In the present and the next sections, we state the main results: description of blow-up solutions by means of solutions shadowing \lq\lq trajectories at infinity".
Key features of blow-up description we shall show here are that, for the desingularized vector field, the convergence to a trajectory on the horizon at a suitably fast rate provides finite-time blow-up, and that the asymptotic blow-up profile is described by the shadowed trajectory, which are extensions of blow-up characterization derived in preceding works (e.g., \cite{Mat2018}).
The shadowing property is geometrically described by {\em asymptotic phase}, which is observed in invariant manifolds with \lq\lq hyperbolic" property.
A suitable class of such manifolds is NHIMs, {\color{black}in which case we can say} that {\em NHIMs on the horizon induce blow-up solutions}.
\par 
Another issue in blow-up description is the magnitude of blow-ups referred to as {\em blow-up rates}. 
In the present argument, we {\em assume convergence to shadowed trajectories with specific rates}, which is stronger than typical convergence properties observed in asymptotic phase.
These convergence rates can relate to spectral properties of associated linearized systems, which are left to Appendix \ref{section-SSspec}, where the correspondence by means of {\em dichotomy spectrum} established in the theory of linear differential systems (e.g., \cite{SS1974, SS1978}) is discussed.
\par
\bigskip
Now let $f: U\to \mathbb{R}^n$ be a smooth (at least $C^1$) asymptotically quasi-homogeneous vector field of type $\alpha = (\alpha_1, \ldots, \alpha_n)$ associated with the set of homogeneity indices $I_\alpha$ with an admissible domain $U\subset \mathbb{R}^n$, and order $k+1$, and consider the (autonomous) system
\begin{equation}
\label{ODE-aut}
{\bf y}' = f({\bf y}).
\end{equation}
Also, let $g$ be the desingularized vector field (\ref{desing-para}) associated with $f$ and the quasi-parabolic embedding $T_{{\rm para};\alpha}$ of type $\alpha$, and $\varphi_g$ be the flow generated by $g$.
The result is also valid for directional embeddings summarized in Appendix \ref{sec:dir}, and possibly other ones with suitable assumptions\footnote{
As far as we understand, other possible choices of compactificaitons are {\em Poincar\'{e}-type} ones (e.g., \cite{Mat2018}) and Poincar\'{e}-Lyapunov disks (e.g., \cite{DH1999}). 
Roughly speaking, the former is the embedding of $\mathbb{R}^n$ into the hemisphere in $\mathbb{R}^{n+1}$, and the latter is the change of variables with quasi-polar coordinates.
}.


First we provide a general description of blow-up solutions by means of solutions shadowing \lq\lq trajectories at infinity".
\begin{thm}[Blow-up: a general description]
\label{thm-blowup-fund}
Suppose that $g$ admits a trajectory $\gamma = \{{\bf x}_\gamma(\tau)\}_{\tau \in \mathbb{R}}\subset \mathcal{E}$ such that, for some ${\bf y}_0\in U$ and the solution ${\bf y}(\tau)$ to (\ref{ODE-aut}) through ${\bf y}_0$ at $\tau = 0$, 
\begin{align}
\label{asymptotic-main}
\|T_{{\rm para};\alpha}({\bf y}(\tau)) - {\bf x}_\gamma(\tau)\| < C e^{-\lambda \tau}
\end{align}
for constants $C > 0$ and $\lambda > 0$.
Then $t_{\max} < \infty$ holds; namely, ${\bf y}(t)$ is a blow-up solution.
\par
If we further assume 
\begin{equation}
\label{decay-precise-NHIM}
\lim_{\tau \to +\infty}e^{(\lambda + \epsilon) \tau} \left(1 - p( T_{{\rm para};\alpha}({\bf y}(\tau))  )^{2c}\right) 
= \begin{cases}
+\infty & \epsilon > 0 \\
0 & \epsilon < 0 
\end{cases},
\end{equation}
then 
\begin{align*}
p({\bf y}(t)) \sim C_0(-\ln (t_{\max}-t)) (t_{\max}-t)^{-1/k}\quad \text{ as }\quad t \to t_{\max}-0
\end{align*}
for some function $C_0(-\ln (t_{\max}-t))$ satisfying
\begin{equation*}
C_0(\tau) = o(e^{\epsilon \tau}),\quad C_0(\tau)^{-1} = o(e^{\epsilon \tau})
\end{equation*}
for any $\epsilon > 0$ as $\tau \to \infty$, equivalently $-\ln (t_{\max}-t) \to +\infty$.
Moreover, we have
\begin{equation*}
y_i(t) \sim C_0(-\ln (t_{\max}-t))^{\alpha_i} x_{\gamma,i} (-\bar c\ln(t_{\max} - t)) (t_{\max}-t)^{-\alpha_i /k} \quad \text{ as }\quad t \to t_{\max}-0
\end{equation*} 
for some constant $\bar c > 0$, provided $x_{\gamma,i} (\tau)\not \to 0$ as $\tau \to \infty$.
\end{thm}

\begin{rem}
The asymptotic behavior (\ref{decay-precise-NHIM}) covers the cases
\begin{equation*}
\lim_{\tau \to +\infty}\frac{ 1 - p( T_{{\rm para};\alpha}({\bf y}(\tau))  )^{2c} }{e^{-\lambda \tau}\tau^m} = c_0 > 0,
\end{equation*}
for some integer $m$, or
\begin{equation*}
0\leq c_1 \leq \frac{ 1 - p( T_{{\rm para};\alpha}({\bf y}(\tau))  )^{2c} }{e^{-\lambda \tau}\tau^m} < c_2
\end{equation*}
for all sufficiently large $\tau > 0$.
\end{rem}

\begin{proof}
Let ${\bf x}(\tau) = T_{{\rm para};\alpha}({\bf y}(\tau))$ in the $\tau$-timescale.
The maximal existence time $t_{\max}$ for ${\bf y}(t)$ in the $t$-timescale is then estimated through (\ref{time-desing-para}) as follows:
\begin{align*}
t_{\max} &\equiv \int_{0}^\infty q({\bf x}(\tau))(1-p({\bf x}(\tau))^{2c})^k d\tau \leq \int_{0}^\infty \{1-p( {\bf x}(\tau) )^{2c}\}^k d\tau
\end{align*}
from the fact that $q({\bf x}) \leq 1$.
On the other hand, because ${\bf x}(\tau)$ shadows ${\bf x}_\gamma(\tau)$ in the sense of (\ref{asymptotic-main}), the quantity $1-p( {\bf x}(\tau) )^{2c}$ can be rewritten as
\begin{align}
\notag
1 - p({\bf x}(\tau))^{2c} &= 1 - \sum_{i\in I_\alpha}x_i(\tau)^{2\beta_i} \\
\notag
	&= \sum_{i\in I_\alpha}x_{\gamma,i}(\tau)^{2\beta_i} - \sum_{i\in I_\alpha}x_i(\tau)^{2\beta_i} \\
\notag
	&= \sum_{i\in I_\alpha} (x_{\gamma,i}(\tau)^2 - x_i(\tau)^2)\left\{ \sum_{j=0}^{\beta_i - 1} x_{\gamma,i}(\tau)^{2(\beta_i-1-j)}x_i(\tau)^{2j} \right\}\\
\label{expansion-1mp}
	&= \sum_{i\in I_\alpha} ( x_{\gamma,i}(\tau) - x_i(\tau) )(x_{\gamma,i}(\tau) + x_i(\tau)) \left\{ \sum_{j=0}^{\beta_i - 1} x_{\gamma,i}(\tau)^{2(\beta_i-1-j)}x_i(\tau)^{2j} \right\}
\end{align}
and, using the inequality
\begin{equation*}
\left\| ( x_{\gamma,i}(\tau) - x_i(\tau) )(x_{\gamma,i}(\tau) + x_i(\tau)) \left\{ \sum_{j=0}^{\beta_i - 1} x_{\gamma,i}(\tau)^{2(\beta_i-1-j)}x_i(\tau)^{2j} \right\} \right\| \leq C_i e^{\lambda \tau} 
\end{equation*}
obtained from (\ref{asymptotic-main}) for some constant\footnote{
Because only components $x_i$ and $x_{\gamma,i}$ with $i\in I_\alpha$ are considered, we have $|x_i(\tau)|, |x_{\gamma,i}(\tau)|\leq 1$ for all $\tau \geq 0$.
} $C_i > 0$, we have
\begin{align*}
t_{\max} &\leq \int_{0}^\infty \{1-p( {\bf x}(\tau) )^{2c}\}^k d\tau
	\leq \tilde C \int_{0}^\infty e^{k\lambda \tau} d\tau < \infty
\end{align*}
for some constant $\tilde C > 0$.
In particular, ${\bf y}(t)$ in the original $t$-timescale is a blow-up solution.
\par
Next we assume (\ref{decay-precise-NHIM}).
For $t$ less than but sufficiently close to $t_{\max}$, we have
\begin{equation*}
t_{\max} - t = \int_{\tau}^\infty q({\bf x}(\tilde \tau)) (1-p({\bf x}(\tilde \tau))^{2c})^k d\tilde \tau
\end{equation*}
with sufficiently large $\tau$, and the asymptotic expression of the right-hand side is our interest here.
From (\ref{decay-precise-NHIM}) we directly have
\begin{align*}
1 - p({\bf x}(\tau))^{2c} 
	&= e^{-\lambda \tau}\bar F(\tau),
\end{align*}
where $\bar F(\tau) = o(e^{\epsilon \tau})$ and $\bar F(\tau)^{-1} = o(e^{\epsilon \tau})$  for any $\epsilon > 0$ as $\tau\to \infty$.
Therefore we have
\begin{equation*}
\frac{1}{k}\ln(t_{\max} - t) = - \lambda \tau + \tilde F(\tau)
\end{equation*}
as $\tau \to \infty$, where $\tilde F(\tau) = o(\tau)$ and $\tilde F(\tau)^{-1} = o(\tau)$ as $\tau \to \infty$.
The above estimates are combined to obtain
\begin{equation*}
p({\bf y}(t)) = \left\{ 1 - \sum_{i \in I_\alpha} x_i(t)^{2\beta_i} \right\}^{-1}\left( \sum_{i \in I_\alpha} x_i(t)^{2\beta_i} \right)^{1/2c} \sim C_0(-\ln (t_{\max}-t))(t_{\max}-t)^{-1/k}
\end{equation*}
as $t\to t_{\max}$, and the proof involving the asymptotic behavior of $p({\bf y}(t))$ is completed.
The last statement follows from the definition of $T_{{\rm para};\alpha}$ (Definition \ref{dfn-quasi-para}) and asymptotic relations we have obtained.
\end{proof}

A usual way to verify assumptions, in particular (\ref{asymptotic-main}), is to construct a collection of trajectories on the horizon guaranteeing the shadowing property in advance.
One specific characterization of such a property is {\em asymptotic phase}. 
Once one can construct invariant sets admitting asymptotic phase with appropriate convergence rate, assumptions involving (\ref{asymptotic-main}) are satisfied and blow-up solutions can be described by means of invariant manifolds for $g$ on the horizon.
For example, when NHIMs on the horizon are constructed, our description is {\color{black}rephrased} as follows, which shall be called {\em normally hyperbolic blow-up}.
Note that other characterizations of asymptotic phase (cf. Remark \ref{rem-asym-phase}) {\color{black}will} yield the similar results.

\begin{cor}[Normally hyperbolic blow-up]
\label{cor-NH-blowup-1}
Suppose that $g$ admits a compact connected NHIM $M \subset \mathcal{E}$ with $\partial M = \emptyset$.
If the solution ${\bf y}(t)$ of (\ref{ODE-aut}) with a bounded initial point ${\bf y}_0 \in \mathbb{R}^n$ whose image ${\bf x}(\tau) = T_{{\rm para};\alpha}({\bf y}(\tau))$ is on $W_{\rm loc}^s(M; \varphi_g)$, then $t_{\max} < \infty$ holds; namely, ${\bf y}(t)$ is a blow-up solution.
Moreover, there is a trajectory $\gamma = \{{\bf x}_\gamma(\tau)\}_{\tau \geq 0}$ for (\ref{desing-para}) on $M\subset \mathcal{E}$ such that (\ref{asymptotic-main}) holds.
\end{cor}

\begin{proof}
Thanks to the stable foliation of $W_{\rm loc}^s(M) \equiv W_{\rm loc}^s(M;\varphi_g)$, there is a trajectory $\gamma = \{{\bf x}_\gamma(\tau) \}_{\tau \geq 0} \subset M$ such that ${\bf x}(0) = T_{{\rm para};\alpha}({\bf y}_0)\in W_{\rm loc}^s ({\bf x}_\gamma(0))$ and ${\bf x}(\tau) = T_{{\rm para};\alpha}({\bf y}(\tau))\in W_{\rm loc}^s ({\bf x}_\gamma(\tau))$ for all $\tau \geq 0$.
The decay estimate (\ref{asymptotic-main}) is the direct consequence of (\ref{estimate-NHIM-exp}).
\end{proof}
As mentioned in the beginning, description of {\em blow-up rates} requires more precise information about convergence rates. 
Details are left to Appendix \ref{section-SSspec}.

\section{Nonautonomous blow-up}
\label{section-blowup-nonaut}

Next we describe blow-up solutions in the {\em nonautonomous system}
\begin{equation}
\label{ODE-non-autonomous}
{\bf y}' = f(t,{\bf y}),\quad {\bf y}(t_0) = {\bf y}_0\in U
\end{equation}
with $t_0\in \mathbb{R}$, an open set $U\subset \mathbb{R}^n$ and a $C^r$-mapping $f:\mathbb{R}\times U\to \mathbb{R}^n$ with $r\geq 1$.
While the essential idea to describe blow-up solutions is the same as the autonomous case, namely embeddings of phase spaces, we need to treat nonautonomous terms appropriately.
First, we have to consider a \lq\lq natural scaling of time $t$" so that dynamics at infinity can be considered in the similar way to autonomous cases.
Second, a natural determination of \lq\lq invariant sets at infinity" for nonautonomous systems is necessary so that the similar ideas to autonomous systems can be applied.
We shall see later that NHIMs will be an appropriate class to describe blow-up solutions in nonautonomous systems.
\par
In the present section, we first propose a natural treatment of the time variable $t$ in our consideration of dynamics at infinity based on {\em extended autonomous systems}, and provide a description of blow-up solutions in a general setting then.
To apply the description to e.g., NHIMs, a geometric technique is introduced so that invariant manifolds with boundary can be treated within the framework of {\em compact, inflowing invariant manifolds}.
Details of the technique are stated in Section \ref{sec-arrangement-NHIM-nonaut}. 
With the help of this technique, blow-ups in the extended autonomous systems can be characterized in the similar way to Section \ref{section-blowup-NHIM}.

\subsection{Desingularized vector fields for nonautonomous systems}
\label{section-dyn-infty-nonaut}

One natural treatment to consider (\ref{ODE-non-autonomous}) is to regard the time variable $t$ as another independent variable depending on an extra variable $\eta$, namely regard (\ref{ODE-non-autonomous}) as the {\em extended autonomous system}:
\begin{equation}
\label{ODE-non-autonomous-extended}
\frac{d{\bf y}}{d\eta} = f(t,{\bf y}),\quad \frac{dt}{d\eta} = 1,\quad {\color{black}(t(0), {\bf y}(0))} = (t_0, {\bf y}_0)\in \mathbb{R}\times U.
\end{equation}
We pay attention to the scaling of the \lq\lq time" variable $t$ so that the extended system is regarded as an asymptotically quasi-homogenenous system.
Because the motion of time does not originally change as solution evolve, it is natural to regard the scaling of $t$ as 
\begin{equation}
\label{scaling-time}
t = \kappa^{0} \tilde t = \tilde t,
\end{equation}
indicating that {\em the time-variable is not scaled depending on the location of trajectories}.
Once this regulation is determined, the machinery provided in Section \ref{section-preliminaries} is applied to (\ref{ODE-non-autonomous-extended}) towards a blow-up description in a similar way to autonomous systems.
We shall describe the desingularized vector field associated with (\ref{ODE-non-autonomous}) under the assertion (\ref{scaling-time}) with the quasi-parabolic embedding.
Application of directional embeddings (Appendix \ref{sec:dir}) is derived in the similar way.
\par
\bigskip
Throughout the rest of this section, given a system (\ref{ODE-non-autonomous}), assume that the extended vector field (\ref{ODE-non-autonomous-extended}) is asymptotically quasi-homogeneous of type $\alpha = (0, \alpha_1, \ldots, \alpha_n)$ and order $k+1$ such that the set $\mathbb{R}\times U$ is admissible with respect to $\alpha$.
Note that the exponent $0$ in the first component of $\alpha$ stems from the scaling (\ref{scaling-time}).
In particular, the associated set of homogeneity indices $I_\alpha$ is necessarily a proper subset of $\{0, 1, \ldots, n\}$.
We then apply the quasi-parabolic embedding in Definition \ref{dfn-quasi-para}. 
In the present case, the functional $p$ in (\ref{func-p}), written as $p(t,{\bf y})$, is independent of $t$ and we shall identify $p(t,{\bf y})$ with $p({\bf y})$.
Using this identification, derive the desingularized vector field introducing
\begin{equation}
\label{time-desing-nonaut}
d\tau = (1-p({\bf x})^{2c})^{-k}q({\bf x})^{-1} d\eta
\end{equation}
in a usual manner, where $q({\bf x})$ is given in (\ref{time-desing-para}).
We then obtain the corresponding desingularized vector field, which can be written by
\begin{equation}
\label{desing-para-nonaut}
\frac{d}{d\tau}\begin{pmatrix}
t \\ {\bf x}
\end{pmatrix} = g(t,{\bf x}) \equiv q({\bf x}) \bar f(t, {\bf x}) - G(t, {\bf x})\bar \Lambda_\alpha \begin{pmatrix}
t \\ {\bf x}
\end{pmatrix}
\end{equation}
with the following notations, which are consistent with the general derivation of desingularized vector fields in autonomous systems:
\begin{align*}
\bar f(t,{\bf x}) &= \begin{pmatrix}
\tilde f_0, \tilde f_1, \ldots, \tilde f_n
\end{pmatrix}^T\quad \text{ with }\quad f_0(t,{\bf y}) = 1,\\
G(t, {\bf x}) &\equiv \sum_{j\in I_\alpha} \frac{x_j^{2\beta_j-1}}{\alpha_j}\tilde f_j(t, {\bf x}),\quad 
\bar \Lambda_\alpha = {\rm diag}(0, \alpha_1, \ldots, \alpha_n).
\end{align*}
Note that, in the above notation, $\tilde f_0(t,{\bf x}) = (1-p({\bf x})^{2c})^k$ via (\ref{f-tilde}).
The above identifications are also consistent with the evolution of $t$ followed by the time-scale desingularization (\ref{time-desing-nonaut}).

\subsection{Fundamental description of blow-ups}

Like autonomous cases, geometric description of blow-ups for (\ref{ODE-non-autonomous-extended}) is based on the construction of invariant manifolds with \lq\lq hyperbolic" structure on the horizon for the desingularized vector field (\ref{desing-para-nonaut}).
On the other hand, $dt / d\tau = 0$ on the horizon $\mathcal{E}= \{p({\bf x}) = 1\}$, and invariant sets on $\mathcal{E}$ in the present case will consist of slices of invariant sets parameterized by the {\em fixed} variable $t$, the first component of phase variables $(t, {\bf x})$.
Therefore, even in the simplest case such as \lq\lq equilibria" on the horizon, they are not isolated.
In this case, \lq\lq hyperbolic structure" will not be expected for invariant sets on the horizon.
Instead, {\em normal}, or {\em partial hyperbolicity} can be naturally considered regarding the evolution in $t$-variable as the tangential direction.
In particular, the concept of normally hyperbolic blow-up stated in Section \ref{section-blowup-NHIM} will be naturally applied in practical problems.
\par
As in the previous section, we only consider quasi-parabolic embeddings.
Fix the corresponding embedding $T_{{\rm para};\alpha}$ associated with $f$.
Let $g$ be the desingularized vector field (\ref{desing-para-nonaut}) associated with $f$ and $T_{{\rm para};\alpha}$.
Also, let $\varphi_g$ be the flow generated by {\color{black}$g(t,{\bf x})$ in (\ref{desing-para-nonaut})}.
Finally, let $\pi_\nu$ be the projection onto the $\nu$-component, $\nu = t, {\bf x}$.
\par
\bigskip
Here we describe blow-up solutions in a general setting.
As seen in Section \ref{section-blowup-NHIM}, our blow-up description; existence and blow-up rates, relies on the existence of {\em asymptotic phase} and {\em precise exponential decay behavior} of trajectory shadowing invariant sets on the horizon for $g$.
Because the desingularized vector field $g(t,{\bf x})$ essentially has the same structure as $g$ in autonomous cases, (\ref{desing-para}), the similar calculations and estimates yield the corresponding description of blow-ups.

\begin{thm}[Nonautonomous blow-up: a general description]
\label{thm-blowup-nonaut-fund}
Suppose that $g$ {\color{black}in (\ref{desing-para-nonaut})} admits a precompact trajectory $\gamma = \{{\bf x}_\gamma(\tau)\}_{\tau \in \mathbb{R}}\subset \mathcal{E}$ such that, for some $(t_0, {\bf y}_0)\in U$ and the solution ${\bf y}(\tau)$ through $(t_0, {\bf y}_0)$ at {\color{black}$\tau = 0$,}
\begin{align}
\label{asymptotic-main-nonaut}
\lim_{\tau \to +\infty} \|\pi_{\bf x} T_{{\rm para};\alpha}(t(\tau), {\bf y}(\tau) ) - {\bf x}_\gamma(\tau)\| = 0.
\end{align}
Then $t_{\max} < \infty$ holds; namely, ${\bf y}(t)$ is a blow-up solution.
\par
If we further assume 
\begin{align}
\label{decay-precise-nonaut}
\lim_{\tau\to +\infty}e^{(\lambda + \epsilon)\tau}\left( 1 - p( T_{{\rm para};\alpha}(t(\tau), {\bf y}(\tau))  )^{2c} \right)
= \begin{cases}
+\infty & \epsilon > 0 \\
0 & \epsilon < 0 
\end{cases}
\end{align}
for a positive constant $\lambda > 0$, 
then 
\begin{align*}
p({\bf y}(t)) \sim C_0(-\ln (t_{\max}-t)) (t_{\max}-t)^{-1/k}\quad \text{ as }\quad t \to t_{\max}-0
\end{align*}
for some function $C_0(-\ln (t_{\max}-t))$ satisfying
\begin{equation*}
C_0(\tau) = o(e^{\epsilon \tau}),\quad C_0(\tau)^{-1} = o(e^{\epsilon \tau})
\end{equation*}
for any $\epsilon > 0$ as $\tau \to \infty$, equivalently $-\ln (t_{\max}-t) \to +\infty$.
Moreover, we have
\begin{equation*}
y_i(t) \sim C_0(-\ln (t_{\max}-t))^{\alpha_i} x_{\gamma,i} (-\bar c\ln(t_{\max} - t)) (t_{\max}-t)^{-\alpha_i /k} \quad \text{ as }\quad t \to t_{\max}-0
\end{equation*} 
for some constant $\bar c > 0$, provided $x_{\gamma,i} (\tau)\not \to 0$ as $\tau \to \infty$.
\end{thm}

\begin{proof}
Observe that $dt / d\tau = 0$ on the horizon $\mathcal{E}$.
From the compactness of the closure ${\rm cl}(\gamma)$ in $\mathcal{E}$, the trajectory $\gamma$ is bounded. 
Hence $\gamma$ is included in $\mathcal{E}\cap \{t = \bar t\}$ for some $\bar t$.
Combining with the convergence (\ref{asymptotic-main}), the maximal existence time $t_{\max}$ of ${\bf y}(t)$ with ${\bf y}(t_0) = {\bf y}_0$ must be equal to $\bar t < \infty$, in particular ${\bf y}(t)$ is a blow-up solution.
The remaining results follow from the same arguments as Theorem \ref{thm-blowup-fund}.
\end{proof}

In the nonautonomous setting, convergence rate of trajectories to shadowed ones is not required because finiteness of $t_{\max}$ is automatically achieved from the boundedness of $\gamma$ and the fact $dt/d\tau = 0$ on $\mathcal{E}$.
The remaining requirements to describe blow-ups are similar to those stated in e.g., {\color{black}Corollary \ref{cor-NH-blowup-1}}.

\subsection{Blow-up descriptions for concrete problems}
\label{sec-arrangement-NHIM-nonaut}

Here we address an intrinsic difficulty so that the first part of our result, Theorem \ref{thm-blowup-nonaut-fund}, can be applied and a technique to overcome it, which ensures the application of the result to a wide range of nonautonomous systems.
Theorem \ref{thm-blowup-nonaut-fund} essentially relies on the existence of asymptotic phase, which are typically ensured for {\color{black}{\em compact}} invariant manifolds.
On the other hand, as seen in simple cases (e.g., Section \ref{section-Pain}), constructed invariant manifolds on $\mathcal{E}$ for the extended system is not compact in general.
Although several generalizations to noncompact invariant manifolds can be applied (e.g., \cite{E2013}), we notice that our attention is blow-up behavior occuring in {\em finite} range of time $t$ (not $\tau$ !), and hence the restriction of invariant manifolds to a compact segment of invariant manifolds is sufficient to consider the concrete behavior of solutions.

We shall pay attention to NHIMs for simplicity, in which case a well-known {\em vector field modification technique on the boundaries}  can be applied to constructing 
inflowing/overflowing invariant manifolds (see Definition \ref{dfn_in_overflow} and comments below).
Its direct application provides invariant foliations, yielding the existence of blow-up solutions in our problems.
As provided below, we can easily construct a modification to obtain both compact inflowing/overflowing invariant manifolds depending on our requirements.
We shall only show a modification for constructing {\em inflowing} invariant manifolds.
Modification for constructing overflowing invariant manifolds is similar\footnote{
Overflowing invariant manifolds on the horizon are applied to detecting blow-up behavior in the {\em reverse} time direction.
}.
\par
\bigskip
For any set $M\subset \widetilde{D}\equiv \mathcal{D}\cup \mathcal{E}$, $\bar t\in \mathbb{R}$ and an interval $I\subset \mathbb{R}$, let
\begin{equation}
\label{M_slice}
M_{\bar t} \equiv M\cap \{t = \bar t\},\quad M_I \equiv M\cap \{t\in I\}
\end{equation}
be the slice and the tube of $M$ on $I$, respectively.
First observe that, when $M \subset \mathcal{E}$ is a compact invariant manifold for $g$ admitting the normally hyperbolic structure, the slice $M_{\bar t}$ is a compact NHIM in $\mathcal{E} \cap \{t = \bar t\}$ because the $dt/d\tau$ always vanishes on $\mathcal{E}$ and hence $t$-evolution plays a role in (non-essential) {\em tangential} evolution for $g$.
Here we choose a collection of compact intervals $I, I'$ satisfying 
\begin{equation}
\label{tuple-intervals}
t_0 \in I \subset {\rm int}\,I' \subset I',
\end{equation}
which are used to modify manifolds over a small interval in $I'\setminus I$ so that dynamics over $I$ is unchanged.

\begin{ass}
\label{ass-nonaut-inv}
Fix an initial time $t_0\in \mathbb{R}$.
There is a collection of compact intervals $I, I'$ satisfying (\ref{tuple-intervals}) and an invariant manifold $M\subset \mathcal{E}$, not necessarily $\partial M \not = \emptyset$, for $g$ in (\ref{desing-para-nonaut}) satisfying the following assertions:
\begin{itemize}
\item $M_{I'}$ admits the normally hyperbolic structure for $\varphi_g$, in particular, there is a continuous splitting
\begin{equation*}
T_{M_{I'}}\mathbb{R}^{n+1} = TM_{I'} \oplus E^u(M_{I'}) \oplus E^s(M_{I'})
\end{equation*}
and, for some constants $C\geq 1$ and $\mu < 0 < \lambda$,
\begin{equation*}
m(D\varphi_g|_{E^u(M_{I'})}) \geq \frac{e^{\lambda t}}{C}\| D\varphi_g|_{TM_{I'}} \|^i,\quad  \| D\varphi_g|_{E^s(M_{I'})} \| \leq Ce^{\mu t} m(D\varphi_g|_{TM_{I'}})^i
\end{equation*}
hold for all $t\geq 0$ and $0 \leq i \leq r$.
\item For any $\bar t\in I'$, the slice $M_{\bar t}$ is nonempty, compact, connected and boundaryless.
In particular, such a slice $M_{\bar t}$ is a NHIM.
\end{itemize}
\end{ass}

\begin{rem}
\label{rem-t0-epsilon}
Without the loss of generality, we may assume that $t_0\in I$ in the above assumption.
Indeed, our interest is asymptotic behavior of solutions ${\bf u}(\tau) = (t(\tau), {\bf x}(\tau))$ for $g$ approaching to $\mathcal{E}$ as $\tau \to \infty$.
If $t_0 = t(\tau_0)\not \in I$, evolve the solution until $\tau_0 + \tilde \tau$ so that $t(\tau_0 + \bar \tau) \in I$ and replacing $t_0$ by $t(\tau_0 + \bar \tau)$, which is the case of our assumption.
\end{rem}


In Assumption \ref{ass-nonaut-inv}, we have a compact (normally hyperbolic) invariant manifold $M_{I'}\subset \mathcal{E}$ with the (nontrivial) invariant boundary.
{\em Modification of vector fields on the boundary} (cf. \cite{W2013}) is well used to guarantee persistence and/or invariant foliations of invariant manifolds.
In our case, let $\rho_{I, I'}: \mathbb{R}\to \mathbb{R}$ be a smooth, nonnegative bump-type function such that
\begin{equation*}
\rho_{I, I'}(t) = \begin{cases}
1 & t\not \in I', \\
0 & t\in I.
\end{cases}
\end{equation*}
We then modify the desingularized vector field (\ref{desing-para-nonaut}) to the following:
\begin{equation}
\label{desing-para-nonaut-mod}
\frac{d}{d\tau}\begin{pmatrix}
t \\ {\bf x}
\end{pmatrix} = \tilde g(t,{\bf x}) \equiv g(t,{\bf x}) + \delta \rho_{I,I'}(t) \begin{pmatrix}
\nu_{I'} \\ 0
\end{pmatrix},
\end{equation}
where $\delta > 0$, and $\nu_{I'}$ is the {\em inward} unit normal vector\footnote{
For constructing overflowing invariant manifolds, modify $\nu_{I'}$ to be the {\em outward} unit normal vector, which ensures the invariant foliation of unstable manifolds (\cite{W2013}).
} of $\partial I'$.
Then we know that $\tilde g = g$ on $M_{I}$ and that trajectories on $\mathcal{E}$ {\em approach to $M_I$} in $t$-direction. 
In other words, for any interval $\tilde I$ satisfying $I \subset {\rm int}\tilde I\subset I'$, the tube $M_{\tilde I}$ is inflowing invariant for the modified vector field $\tilde g$.
Therefore, according to the general theory of NHIMs, the inflowing invariant manifold $M_{\tilde I}$ for $\tilde g$ admits the local stable manifold $W^s_{\rm loc}(M_{\tilde I}; \tilde g)$ and its invariant foliation (e.g., \cite{W2013}) such that the restriction of the base space of $W^s_{\rm loc}(M_{\tilde I}; \tilde g)$ to $M_I$ provides the invariant foliation of $W^s_{\rm loc}(M_{I}; g)$; for {\em the original vector field} $g$ in $\{t\in I\}$.
As a result, blow-up solutions for (\ref{ODE-non-autonomous-extended}) stated in Theorem \ref{thm-blowup-nonaut-fund} can be described by means of $M_I$.

\begin{cor}[Nonautonomous blow-up: the existence by means of invariant manifolds]
\label{cor-nonaut-blowup-1}
Fix an initial time $t_0 \in \mathbb{R}$.
Suppose that $g$ admits an invariant manifold $M \subset \mathcal{E}$ satisfying all requirements in Assumption \ref{ass-nonaut-inv} with compact intervals $I, I'\subset \mathbb{R}$.
Also assume that the solution ${\bf y}(t)$ of (\ref{ODE-non-autonomous-extended}) with a bounded initial point ${\bf y}_0 (= {\bf y}(t_0)) \in \mathbb{R}^n$ whose image ${\bf x}(\tau) = T_{{\rm para};\alpha}({\bf y}(\tau))$ for $T_{{\rm para};\alpha}$ is on $W_{\rm loc}^s(M_I; g)$, where $\tau_0$ is the corresponding time of $t_0$ for the time-scale desingularization.
Then $t_{\max} < \infty$ and (\ref{asymptotic-main-nonaut}) holds with a trajectory $\gamma \subset M_I$.
\end{cor}

We leave the treatment of blow-up rates in Appendices \ref{section-SSspec} and \ref{section-technical_nonaut} because several additional techniques stemming from {\color{black}{\em spectral theory of linear differential systems} and {\em differential topology} for achieving specific convergence rates}.
Here we only provide a special case, which can be simply applied to various nonautonomous systems such as  examples in Section \ref{section-blowup-ex}.

\begin{cor}[Blow-up rates for nonautonomous blow-up: a special case]
\label{cor-nonaut-blowup}
In addition to settings in Assumption \ref{ass-nonaut-inv}, suppose that 
$f$ is $C^4$ and that
\begin{align}
\notag
&M_{I'} = \{ (t, {\bf x}_\ast(t)) \mid t\in I',\, {\bf x}_\ast(t) \text{ is an equilibrium satisfying (\ref{NH-spec-nonaut-stationary}) below}\\
\notag
	&\quad \quad \quad \quad {\color{black}\text{and admitting (\ref{asymptotic-main-nonaut})}} \}, \\
\label{NH-spec-nonaut-stationary}
&\sharp \{{\rm Spec}(Dg(t, {\bf x}_\ast(t))) \cap i\mathbb{R} \} = 1\quad \text{ for all }\quad t\in I'
\end{align}
where ${\rm Spec}(A)$ denotes the set of eigenvalues of a squared matrix $A$, 
that is, $M_{I'}$ is a curve of hyperbolic equilibria parameterized by $t$.
{\color{black}Finally, suppose that \lq\lq the Sternberg-Sell condition of order $1$" (Definition \ref{dfn-Sternberg-Sell}) is satisfied}. 
Then the asymptotic estimate (\ref{decay-precise-nonaut}) holds true with ${\bf x}_\gamma(\tau) \equiv {\bf x}_\ast(t)$ for each $t\in I$.
In particular, asymptotic behavior for blow-up solutions stated in Theorem \ref{thm-blowup-nonaut-fund} is obtained.
\end{cor}
\begin{proof}
See Sections \ref{section-technical_nonaut} and \ref{section-SSspec}.
\end{proof}



\section{Examples}
\label{section-blowup-ex}

Examples showing the applicability of our results are collected {\color{black}here}.
In addition to the applicability, we aim at seeing that {\em the present methodology unifies the approach to finite-time singularities in nonautonomous problems}.
The examples we shall show are highlighted as follows:
\begin{itemize}
\item Section \ref{section-Pain} addresses the first Painlev\'{e} equation, which the existence of singular solutions is well-known in the treatment of {\em regular singular points}.
\item Section \ref{section-KK} relates to a regularized system for exhibiting {\em singular shocks} in (non-)hyperbolic systems of conservation laws.
\item Section \ref{section-ss} addresses {\em self-similar profiles} of nonlinear diffusion equations with finite-time singularity.
\item Section \ref{section-MEMS} addresses {\em radially symmetric profiles} of a reaction-diffusion equation with negative power nonlinearity.
\end{itemize}
All examples were investigated individually in preceding studies, and independent techniques were required to unravel the comprehensive morphology of solutions.
When we pay our attention to blow-ups, the proposed methodology unravels all these features in a systematic way, indicating that a series of machineries involving embeddings of phase spaces as well as time-scale desingularizations can unify descriptions of finite-time singularities.

\begin{rem}
In all examples, invariant manifolds in our interest are collections of equilibria.
Therefore Corollary \ref{cor-nonaut-blowup} are mainly applied.
In particular, the Sternberg-Sell condition of order $1$ (Definition \ref{dfn-Sternberg-Sell}) is satisfied in all examples.
\end{rem}

\begin{rem}
As we shall see in Section \ref{section-Pain}, as well as commented in \cite{LMT2023}, arguments through global (parabolic-type or Poincar\'{e}-type, say \cite{Mat2018}) embeddings require lengthy and tedious calculations, although we do not suffer from any change of charts\footnote{
Advantages that quasi-parabolic embeddings are used as our central machinery are discussed in \cite{asym2, LMT2023}.
}.
This is the reason why we apply another embeddings, called {\em directional embeddings} briefly summarized in Appendix \ref{sec:dir}, to the remaining examples.
Note that all results in Sections \ref{section-blowup-NHIM} and \ref{section-blowup-nonaut} are essentially valid.
See e.g. \cite{Mat2018} for the correspondence of arguments.
\end{rem}

\subsection{The first Painlev\'{e} equation}
\label{section-Pain}

First we consider the first Painlev\'{e} equation 
\begin{equation}
\label{Pain1-original}
u'' = 6u^2 + t,\quad ' = \frac{d}{dt}
\end{equation}
from the viewpoint of blow-up description.
It is well-known that (\ref{Pain1-original}) possesses the following solution: for fixed $t_0\in \mathbb{R}$, 
\begin{equation}
\label{blow-up-Pain1}
u(t) \sim 6(t_{\max} - t)^{-2}\quad \text{ as }\quad t \to t_{\max}-0
\end{equation}
for some $t_{\max} = t_{\max}(t_0) > t_0$ and sufficiently large initial point $u_0 = u(t_0)$.
Although this asymptotic behavior can be derived via substitution of a formal (Frobenius-type) power series solution into the equation (e.g., \cite{KJH1997}), we shall derive the existence of such a solution through dynamics at infinity.
\par
Rewrite (\ref{Pain1-original}) as the system of the first order ODEs:
\begin{equation}
\label{Pain1}
\begin{cases}
\chi' = 1, & \\
u' = v, & \\
v' = 6u^2 + \chi, &
\end{cases}\quad ' = \frac{d}{dt}.
\end{equation}
We immediately have the following property.
\begin{lem}
The system (\ref{Pain1}) is asymptotically quasi-homogeneous of type $(0,2,3)$ and order $k+1 = 2$.
\end{lem}
In the present example, we apply the quasi-parabolic embedding
\begin{equation*}
\chi = \chi,\quad u = \frac{x_1}{(1-p({\bf x})^{6})^2},\quad v = \frac{x_2}{(1-p({\bf x})^{6})^3}, 
\end{equation*}
where
\begin{equation*}
{\bf x} = (x_1, x_2)^T,\quad p({\bf x})^6 = x_1^6 + x_2^4
\end{equation*}
with $(\beta_1, \beta_2) = (3,2)$ and $c = 6$. 
Following the general derivation (Section \ref{section-preliminaries}), let
\begin{align*}
&\kappa := (1-p({\bf x})^{6})^{-1},\quad \tilde f_1(\chi, {\bf x}) := x_2,\quad \tilde f_2(\chi, {\bf x}) := 6x_1^2 + \kappa^{-4} \chi,\\
&G(\chi, {\bf x}) := \frac{x_1^5}{2} \tilde f_1(\chi, {\bf x}) + \frac{x_2^3}{3}\tilde f_2(\chi, {\bf x}) = \frac{x_1^5}{2}x_2 + \frac{x_2^3}{3}(6x_1^2 + \kappa^{-4} \chi).
\end{align*}
Under the time-scale desingularization
\begin{equation*}
\frac{d\tau}{dt} = \kappa^{-1},
\end{equation*}
we have the desingularized vector field (\ref{desing-para-nonaut}) with
$c=6$, $k=1$ and $\bar \Lambda_\alpha = {\rm diag}(0,2,3)$.
\par
The next step is to find invariant sets on the horizon $\mathcal{E} = \{p({\bf x})=1\}$.
On $\mathcal{E}$, equilibria satisfy $g_1(\chi, {\bf x}) = g_2(\chi, {\bf x}) = 0$, where
\begin{equation*}
g_1(\chi, {\bf x}) = x_2(1-x_1^6 - 4x_1^3 x_2^2),\quad g_2(\chi, {\bf x}) = \frac{3}{2}x_1^2 (4-x_1^3x_2^2 - 4x_2^4)
\end{equation*}
are the second and the third components of the desingularized vector field, respectively.
Note that $\dot \chi = 0$ is automatically satisfied on $\mathcal{E}$.
We find, for each $\chi \in \mathbb{R}$, the points $(\chi, {\bf x}_\ast)$, where 
\begin{align*}
&{\bf x}_\ast = (x_{1; +}, x_{2; +}),\quad (x_{1; +}, x_{2; -}),\quad (x_{1; -}, x_{2; +}),\, \text{ or }\, (x_{1; -}, x_{2; -}),\\
&x_{1; \pm} = \pm \frac{1}{17^{1/6}},\quad x_{2; \pm} = \pm \frac{2}{17^{1/4}}
\end{align*}
are equilibria on $\mathcal{E}$.
\par
Now we pay attention to the family of equilibria $M_{++} = \{(\chi, x_{1; +}, x_{2; +})\}$.
The Jacobian matrix of (\ref{desing-para-nonaut}) at each point on $M_{++}$ is
\begin{equation*}
\begin{pmatrix}
0 & \ast \\
O & \tilde A
\end{pmatrix},\quad \tilde A = 
\begin{pmatrix}
\frac{11}{2}x_1^5 \tilde f_1|_{\mathcal{E}} + \left.\frac{\partial \tilde f_1}{\partial x_1}\right|_{\mathcal{E}} - 2\left. \left( G + x_1 \frac{\partial G}{\partial x_1}\right) \right|_{\mathcal{E}} 
& \frac{11}{3}x_2^3 \tilde f_1|_{\mathcal{E}} + \left.\frac{\partial \tilde f_1}{\partial x_2}\right|_{\mathcal{E}} - \left. 2 x_1 \frac{\partial G}{\partial x_2} \right|_{\mathcal{E}} \\
\frac{11}{2}x_1^5 \tilde f_2|_{\mathcal{E}}  + \left.\frac{\partial \tilde f_2}{\partial x_1}\right|_{\mathcal{E}} - \left. 3x_2 \frac{\partial G}{\partial x_1} \right|_{\mathcal{E}} 
& \frac{11}{3}x_2^3 \tilde f_2|_{\mathcal{E}}  + \left.\frac{\partial \tilde f_2}{\partial x_2}\right|_{\mathcal{E}} - 3\left. \left( G + x_2 \frac{\partial G}{\partial x_2}\right) \right|_{\mathcal{E}} 
\end{pmatrix},
\end{equation*}
where
\begin{align*}
&\tilde f_1|_{\mathcal{E}} = x_2|_{\mathcal{E}},\quad \tilde f_2|_{\mathcal{E}} = 6x_1^2|_{\mathcal{E}},\quad
\left.\frac{\partial \tilde f_1}{\partial x_1}\right|_{\mathcal{E}}  = 0,\quad \left.\frac{\partial \tilde f_1}{\partial x_2}\right|_{\mathcal{E}} = 1,\quad \left.\frac{\partial \tilde f_2}{\partial x_1}\right|_{\mathcal{E}} = 12x_1,\quad \left.\frac{\partial \tilde f_2}{\partial x_2}\right|_{\mathcal{E}} = 0,\\
&G|_{\mathcal{E}} = \frac{x_1^5}{2} \tilde f_1|_{\mathcal{E}} + \frac{x_2^3}{3}\tilde f_2|_{\mathcal{E}} = \left.\left(\frac{1}{2}x_1^5 x_2 + 2x_1^2 x_2^3\right)\right|_{\mathcal{E}},\\
&\left. \frac{\partial G}{\partial x_1} \right|_{\mathcal{E}} = \frac{5}{2}x_1^4 \tilde f_1|_{\mathcal{E}}  + \frac{x_1^5}{2}\left. \frac{\partial \tilde f_1}{\partial x_1} \right|_{\mathcal{E}} + \frac{x_2^3}{3}\left. \frac{\partial \tilde f_2}{\partial x_1} \right|_{\mathcal{E}} = \left.\left(\frac{5}{2}x_1^4 x_2 + 4x_1 x_2^3\right)\right|_{\mathcal{E}},\\
& \left. \frac{\partial G}{\partial x_2} \right|_{\mathcal{E}} = \frac{x_1^5}{2}\left. \frac{\partial \tilde f_1}{\partial x_2} \right|_{\mathcal{E}} +x_2^2 \tilde f_2|_{\mathcal{E}} + \frac{x_2^3}{3}\left. \frac{\partial \tilde f_2}{\partial x_2} \right|_{\mathcal{E}} = \left.\left(\frac{1}{2}x_1^5 + 6x_1^2 x_2^2\right)\right|_{\mathcal{E}},
\end{align*}
and several terms being identically zero on $\mathcal{E}$ are omitted.
Note that the above matrix is independent of $\chi$, namely the flow and its linearization are autonomous on $\mathcal{E}$.
Substituting $(x_1, x_2) = (x_{1; +}, x_{2; +})$ into the above matrix, we have
\begin{align*}
\frac{11}{2}x_1^5 \tilde f_1|_{\mathcal{E}} + \left.\frac{\partial \tilde f_1}{\partial x_1}\right|_{\mathcal{E}} - 2\left. \left( G + x_1 \frac{\partial G}{\partial x_1}\right) \right|_{\mathcal{E}} 
%
	 &= \frac{11}{17^{13/12}} - \frac{108}{17^{13/12}}  = - \frac{97}{17^{13/12}},\\
\frac{11}{3}x_2^3 \tilde f_1|_{\mathcal{E}} + \left.\frac{\partial \tilde f_1}{\partial x_2}\right|_{\mathcal{E}} - \left. 2 x_1 \frac{\partial G}{\partial x_2} \right|_{\mathcal{E}} 
	&= \frac{11\cdot 16}{51} + 1 - \frac{49}{17} 
	= \frac{80}{51},\\
\frac{11}{2}x_1^5 \tilde f_2|_{\mathcal{E}} + \left.\frac{\partial \tilde f_2}{\partial x_1}\right|_{\mathcal{E}} - \left. 3x_2 \frac{\partial G}{\partial x_1} \right|_{\mathcal{E}} 
	&= \frac{33}{17^{7/6}} + \frac{12\cdot 17}{17^{7/6}}  - \frac{6}{17^{7/6}} \left( 5 + 32 \right) 
	= \frac{15}{17^{7/6}},
\end{align*}
\begin{align*}
&\frac{11}{3}x_2^3 \tilde f_2|_{\mathcal{E}} + \left.\frac{\partial \tilde f_2}{\partial x_2}\right|_{\mathcal{E}} - 3\left. \left( G + x_2 \frac{\partial G}{\partial x_2}\right) \right|_{\mathcal{E}} 
%
	= \frac{176}{17^{13/12}} - \frac{3}{17^{13/12}} \left(1 + 16 + 1 + 48\right)  = -\frac{22}{17^{13/12}}
\end{align*}
and hence the matrix $\tilde A$ is
\begin{equation*}
\tilde A = 
\begin{pmatrix}
- 97 / 17^{13/12} & 80 / 51 \\
15 / 17^{7/6} & - 22 / 17^{13/12}
\end{pmatrix}
\equiv \begin{pmatrix}
a & b \\
c & d
\end{pmatrix}.
\end{equation*}
%
Now
\begin{align*}
&a+d = -\frac{119}{17^{13/12}},\quad (a+d)^2 = \frac{119^2}{17^{13/6}},\quad 
ad-bc 
= \frac{97\cdot 22 - 400}{17^{13/6}},\\
&(a+d)^2 - 4(ad-bc) = \frac{119^2 - 4(97\cdot 22 - 400)}{17^{13/6}} 
= \frac{4025}{17^{13/6}}
\end{align*}
and hence the eigenvalues of $\tilde A$ are
\begin{equation*}
\lambda = \frac{1}{2} \left\{ (a+d) \pm \sqrt{ (a+d)^2 -4(ad-bc)} \right\} = \frac{-119 \pm \sqrt{7225}}{2\cdot 17^{13/12}} = \frac{-119 \pm 85}{2\cdot 17^{13/12}} = \frac{-1}{17^{1/12}},\quad \frac{-6}{17^{1/12}},
\end{equation*}
indicating that $(x_{1; +}, x_{2; +})$ is attracting (sink-type), and hence the invariant manifold $M_{++}$ is normally attracting.
In particular, for any compact interval $I\subset \mathbb{R}$, the restriction $M_{++;I} \equiv M_{++}\cap \{t\in I\}$ is a NAIM satisfying all requirement in Corollary \ref{cor-nonaut-blowup}.
\par
We therefore see that, for particular choice of an interval $I \subset \mathbb{R}$, solutions ${\bf x}(\tau) \subset W^s_{\rm loc}(M_{++;I})$ characterize blow-up solutions of (\ref{Pain1}) with the asymptotic behavior $O((t_{\max} - t)^{-2})$ as $t\to t_{\max}-0$.
This result confirms and guarantees of the existence of blow-up solutions of the form (\ref{blow-up-Pain1}) {\em without any algebraic trial-and-error arguments}.

\begin{thm}[cf. \cite{KJH1997}]
Consider (\ref{Pain1-original}) with the initial time $t_0\in \mathbb{R}$ and $u(t_0) = u_0\in \mathbb{R}$.
If $u_0$ is sufficiently large, the solution $u(t) = u(t;t_0, u_0)$ with $u(t_0) = u_0$ blows up at a time $t_{\max} = t_{\max}(t_0) < \infty$ with
\begin{equation*}
u(t) \sim c_u(t_{\max} - t)^{-2} \quad \text{ as }\quad t\to t_{\max}-0
\end{equation*}
for some constant $c_u>0$.
\end{thm}



\subsection{A problem in Dafermos profiles for singular shocks}
\label{section-KK}
The next example is the following system with $0 \leq \epsilon \ll 1$ as a parameter controlling the multi-scale structure:
\begin{equation}
\label{KK}
\begin{cases}
\chi' = \epsilon, & \\
u_1' = u_1^2 - u_2 - \chi u_1 - w_1, & \\
u_2' = \frac{1}{3}u_1^3 - u_1 - \chi u_2 - w_2, & \\
w_1' = -\epsilon u_1, & \\
w_2' = -\epsilon u_2, & \\
\end{cases}\quad ' = \frac{d}{dt}.
\end{equation}

\begin{rem}
\label{rem-sing-shock-mech}
The origin of (\ref{KK}) is the system of conservation laws
\begin{equation}
\label{KK-PDE}
(u_1)_t + \left(u_1^2 - u_2\right)_x = 0,\quad (u_2)_t + \left(\frac{1}{3}u_1^3 - u_1\right)_x = 0\quad \Leftrightarrow \quad U_t + F(U)_x = 0,
\end{equation}
proposed by Keyfitz and Kranser (e.g., \cite{KK1995, KK1990, K2011}. See also \cite{SSS1993, SS2004}) with identification
\begin{equation*}
U(t,x) = (u_1(t,x),u_2(t,x))^T,\quad F(U) = \left(u_1^2 - u_2, \frac{1}{3}u_1^3 - u_1\right)^T.
\end{equation*}
Applying the {\em Dafermos regularization} (e.g., \cite{S2004}) to (\ref{KK-PDE}), we have
\begin{equation}
\label{KK-PDE-Dafermos}
U_t + F(U)_x = \epsilon t U_{xx},
\end{equation}
and applying the ansatz of self-similar profiles, called the {\em Dafermos profiles}: $U(t,x) = U^{\epsilon}(x/t)$, we obtain the extended autonomous system
\begin{equation*}
\chi' = \epsilon,\quad U' = V,\quad V' = \{ DF(U) - \chi I\}V, \quad ' = \frac{d}{dt},
\end{equation*}
where $U=(u_1, u_2)^T$ and $V = (v_1, v_2)^T$.
An alternative form is derived through the change of variables (e.g., \cite{S2004})
\begin{equation*}
(U,V,\chi) \mapsto (U,W,\chi) = (U, F(U) - \chi U - V, \chi),
\end{equation*}
which is exactly (\ref{KK}), where $W = (w_1, w_2)^T$.
In \cite{S2004}, this system is considered to construct the Dafermos profile of {\em singular shocks}, namely shock fronts with Dirac's delta-type singularity on the front.
A novel observation in singular shocks is that these shocks can be constructed as the singular limit of singularly perturbed solutions admitting arbitrarily large magnitudes.
The corresponding trajectories are constructed by the {\em connecting orbits of equilibria at infinity} under technical change of coordinates.
The present example as well as preceding works by the author (e.g., \cite{Mat2018}) is inspired by reconstructing such singular objects within standard and unified arguments of dynamics at infinity.
\end{rem}

Based on arguments in preceding works, we have the following observation.

\begin{lem}[cf. \cite{Mat2018, S2004}]
The system (\ref{KK}) is asymptotically quasi-homogeneous of type $(0,1,2,1,2)$ and order $k+1 = 2$.
\end{lem}
We apply the directional embedding $T_{d;\alpha}: (\chi, u_1, u_2, w_1, w_2)\mapsto (\chi, x_1, s, x_3, x_4)$ defined by
\begin{equation}
\label{dir-cpt-KK}
\chi = \chi, \quad u_1 = \frac{x_1}{s},\quad u_2 = \frac{1}{s^2},\quad w_1 = \frac{x_3}{s},\quad w_2 = \frac{x_4}{s^2}.
\end{equation}

\begin{rem}
The present choice of the localizing direction, namely $u_2 = s^{-2}$, stems from the previous arguments for deriving blow-up behavior of dynamical systems with nonlinearity extracted from (\ref{KK-PDE}) (cf. \cite{LMT2023, Mat2018, MT2020_1}):
\begin{equation*}
u_1' = u_1^2 - u_2,\quad u_2' = \frac{1}{3}u_1^3 - u_1.
\end{equation*}
It is observed in the above references that the corresponding desingularized vector fields through global embeddings like quasi-parabolic ones admit totally four equilibria on the horizon, all of which correspond to blow-up solutions (in the positive or negative time direction) such that $v$ goes to \lq\lq positive" infinity.
See references mentioned for details.
This is the reason why we have chosen the embedding in the direction specified as (\ref{dir-cpt-KK}).
\end{rem}

Using the change of coordinates (\ref{dir-cpt-KK}) as well as the time-scale desingularization
\begin{equation*}
\frac{d\tau}{dt} = s^k = s,
\end{equation*}
we obtain the associated desingularized vector field:
\begin{equation}
\label{KK-desing}
\begin{cases}
\dot \chi = \epsilon s, & \\
\dot x_1 = x_1\tilde g_0  + (x_1^2 - 1) - s(\chi x_1 + x_3),\\
\dot s = s \tilde g_0, & \\
\dot x_3 = x_3\tilde g_0 -\epsilon x_1 s,\\
\dot x_4 = 2x_4\tilde g_0 -\epsilon s,\\
\end{cases}
\quad \dot {} = \frac{d}{d\tau},\quad \tilde g_0 = -\frac{1}{6}x_1^3 + \frac{x_1}{2}s^2 + \frac{s}{2}(\chi + x_4).
\end{equation}
We easily see that the horizon $\{s=0\}$ admits four families of equilibria
\begin{equation*}
(\chi, x_1, s, x_3, x_4) =  \left(\chi, \pm \sqrt{3\pm \sqrt{3}}, 0, 0, 0 \right).
\end{equation*}
Let 
\begin{equation}
\label{eq-KK-original}
p_{1, \pm} := \pm \sqrt{3+\sqrt{3}}, \quad  p_{2, \pm} := \pm \sqrt{3-\sqrt{3}}.
\end{equation}
The Jacobian matrices on the above family of equilibria are
\begin{equation*}
\begin{pmatrix}
0 & 0 & \epsilon & 0 & 0 \\
0 & -\frac{2}{3}p_{i, \pm}^3 + 2p_{i, \pm} & -\frac{1}{2}\chi p_{i, \pm} & 0 & 0\\
0 & 0 & -\frac{1}{6}p_{i, \pm}^3 & 0 & 0\\
0 & 0 & -\epsilon p_{i, \pm} & -\frac{1}{6}p_{i, \pm}^3 & 0\\
0 & 0 & -\epsilon & 0 & -\frac{1}{6}p_{i, \pm}^3  \\
\end{pmatrix}
\end{equation*}
For simplicity, we further assume that $\epsilon = 0$ for a while. 
Then, in the {\em singular fast limit system}, namely (\ref{KK-desing}) with $\epsilon = 0$, four families of equilibria are invariant with normally hyperbolic structure.  
Among four points (\ref{eq-KK-original}), $p_{1,+}$ and $p_{2,+}$ are those where the flow is attracting towards the horizon.
Moreover, $p_{1,+}$ is a saddle, which is our main issue\footnote{
According to the mechanism of singular shocks mentioned in Remark \ref{rem-sing-shock-mech}, {\em saddle}-like invariant sets would be considered to play a key role in constructing these shocks.
}.
From our construction and the above observation, the family
\begin{equation*}
M_{1,+} = \{(\chi, p_{1,+}, 0, 0, 0) \mid \chi \in \mathbb{R}\}
\end{equation*}
determines a (noncompact) NHIM.
Because $\chi$ is just a a parameter on the horizon $\{s = 0\}$, for any compact interval $I\subset \mathbb{R}$, the restriction
\begin{equation*}
M_{1,+;I} = \{(\chi, p_{1,+}, 0, 0, 0) \mid \chi \in I\}
\end{equation*}
becomes a compact NHIM (with boundary).
In particular, all requirements in Corollary \ref{cor-nonaut-blowup} are satisfied.
By the geometric singular perturbation theory (cf. \cite{F1979}) (with boundary modification), the above conclusion also holds for all $0\leq \epsilon \ll 1$.
\par
We therefore conclude that, for sufficiently large $I$ depending on the initial time $t_0$, the local stable manifold $W_{\rm loc}^s(M_{1,+;I})$ induce a blow-up solution.
More precisely, there is an equilibrium 
$p_\ast = (\chi_\ast, p_{1,+}, 0,0,0)\in M_{1,+;I}$ for some $\chi_\ast = \chi_\ast(t_0)$ such that the trajectory on $W_{\rm loc}^s(p_\ast)$ induces a blow-up solution with the asymptotic behavior
\begin{equation*}
u_1(t) = \frac{x_1(t)}{s(t)}\sim c(t_{\max} - t)^{-1}\quad \text{ as }\quad t\to t_{\max}-0
\end{equation*}
for some constant $c > 0$.
\begin{thm}[cf. \cite{KK1995, S2004}]
Consider (\ref{KK}) and a solution $(\chi, u_1, u_2, w_1, w_2)(t)$ with $0\leq \epsilon \ll 1$.
Let $g$ be the desingularized vector field (\ref{KK-desing}).
Assume that, under the directional embedding $T_{d;\alpha}$ given in (\ref{dir-cpt-KK}), $T_{d;\alpha}(\chi, u_1, u_2, w_1, w_2)(\bar \tau) \in W_{\rm loc}^s(p_\ast)$ holds for some $\bar \tau$.
Then $(\chi, u_1, u_2, w_1, w_2)(t)$ is a blow-up solution with
\begin{equation*}
u_1(t) \sim c_1(t_{\max}- t)^{-1},\quad u_2(t) \sim c_2(t_{\max}- t)^{-2} \quad \text{ as }\quad t\to t_{\max}-0
\end{equation*}
for some constants $c_1, c_2 > 0$.
\end{thm}

In the above theorem we do not mention the asymptotic behavior of $w_1$ and $w_2$ because the corresponding $x$-component of $p_\ast$ is $0$, in which case further terms have to be calculated for these asymptotic behavior within the present approach (cf. \cite{asym1}).

\subsection{A system associated with self-similarity}
\label{section-ss}

The next example we shall consider is the following system:
\begin{equation}
\label{ss-original}
( u^{m-1}u' )' + \beta \chi u' + \alpha u = 0,\quad ' = \frac{d}{d\chi},\quad \chi \in \mathbb{R},
\end{equation}
where $\alpha, \beta \in \mathbb{R}$ are parameters. 
The system (\ref{ss-original}) originates from the diffusion equation (e.g., \cite{FV2003})
\begin{equation*}
U_t = (U^{m-1} U_x)_x.
\end{equation*}
The parameter $m$ controls the strength of nonlinear diffusion.
Paying our attention to {\em self-similar solutions} of the form\footnote{
The ansatz (\ref{backward-ss}) represents the {\em backward} self-similarity.
Although there are another types of self-similarity; {\em forward} and {\em exponential-type} ones, the governing equation becomes the same one, (\ref{ss-original}).
} 
\begin{equation}
\label{backward-ss}
U(t,x) = (T-t)^\alpha u(x(T-t)^{\beta}),\quad t < T
\end{equation}
for some $T>0$, the system is reduced to (\ref{ss-original}).
We concentrate on the very fast diffusion case; $m < 0$, and assume that $\beta < 0$.
In what follows we investigate blow-up solutions following our proposed machinery.
First rewrite (\ref{ss-original}) as the first order extended autonomous system:
\begin{equation}
\label{ss}
\begin{cases}
\chi' = 1, & \\
u' = u^{1-m} v, & \\
v' = -\beta \chi u^{1-m}v - \alpha u, & \\
\end{cases}\quad ' = \frac{d}{dt}.
\end{equation}
Direct calculations yield the following observation.

\begin{lem}
The system (\ref{ss}) is asymptotically quasi-homogeneous of type $(0,1,1)$ and order $k+1 = 2-m$.
\end{lem}

We then apply the directional embedding of type $(0,1,1)$:
\begin{equation}
\label{dir-cpt-ss}
\chi = \chi, \quad u = \frac{1}{s},\quad v = \frac{x}{s}.
\end{equation}
Using this change of coordinates as well as the time-scale desingularization
\begin{equation*}
\frac{d\tau}{dt}= s^{1-m},
\end{equation*}
we obtain the associated desingularized vector field:
\begin{equation}
\label{ss-desing}
\begin{cases}
\dot \chi = s^{1-m}, & \\
\dot s = -sx, & \\
\dot x = -x^2 - \beta \chi x - \alpha s^{1-m}, & \\
\end{cases}\quad \dot {} = \frac{d}{d\tau}.
\end{equation}
We easily see that the horizon $\{s=0\}$ admits a family of equilibria
\begin{equation*}
\{(\chi, 0,0)\},\quad \text{ and }\quad \{(\chi, 0, -\beta \chi)\}.
\end{equation*}
We further investigate the dynamics around the invariant manifold $M_{\geq \chi_0} = \{(\chi, 0, -\beta \chi)\in \mathbb{R}^3 \mid  \chi \geq \chi_0\}$ with $\chi_0 > 0$.
Thanks to the assumption $m < 0$, the linearized matrix at each point on $M_{\geq \chi_0}$ becomes very simple:
\begin{equation*}
\begin{pmatrix}
0 & 0 & 0\\
0 & \beta \chi & 0\\
\beta^2 \chi & 0 & \beta \chi
\end{pmatrix},
\end{equation*}
which indicates that, for any fixed $\chi_1 \in (\chi_0, \infty)$, the invariant manifold $M_{[\chi_0, \chi_1]} \equiv M_{\geq \chi_0}\cap \{\chi \in [\chi_0, \chi_1]\}$ is invariant with normally attracting structure.
{\color{black}In particular,} Corollary \ref{cor-nonaut-blowup} can be applied to (\ref{ss}), yielding the following result.
The rate we have obtained coincides with the preceding result shown in \cite{FV2003}.

\begin{thm}[cf. \cite{FV2003}]
\label{thm-FV}
Assume that $m < 0$ and $\beta < 0$.
Then the solution $(\chi, u(\chi),v(\chi))$ through the initial point $(\chi_0, u_0, v_0)$ with $\chi_0 > 0$ whose image under the directional embedding (\ref{dir-cpt-ss}) is on the (local) stable manifold of $M_{\geq \chi_0}$ for (\ref{ss-desing}) is a blow-up solution with the blow-up rate $u(\chi) \sim c_u(\chi_{\max} - \chi)^{-1/(1-m)}$ and $v(\chi) \sim c_v(\chi_{\max} - \chi)^{-1/(1-m)}$ for some constants $c_u, c_v > 0$.
\end{thm}

\begin{rem}
In \cite{FV2003}, dynamics of self-similar profiles is studied under a technical change of variables, which changes (\ref{ss-original}) to an autonomous system.
More precisely, (\ref{ss-original}) with $u(0) = 1$ and $u'(0) = a \in \mathbb{R}$ 
is transformed via
\begin{equation*}
\chi = e^r,\quad X(r) = \frac{\chi u'}{u},\quad Y(r) = \chi^2 u^{1-m}
\end{equation*}
into
\begin{equation*}
\begin{cases}
\dot X = X- mX^2 - (\alpha + \beta X)Y, & \\
\dot Y = (2+(1-m)X)Y, &
\end{cases}\quad 
\dot X = dX/dr.
\end{equation*}
In our methodology, on the other hand, no additional techniques are introduced to obtain the same result when we pay our attention to blow-up solutions.
\par
We have assumed $\beta < 0$ in the present study.
We can easily check that the family of equilibria $\{(\chi, 0, -\beta \chi)\}$ for (\ref{ss-desing}) on the horizon undergoes a transcritical bifurcation at $\beta = 0$. 
In particular, when $\beta > 0$, the invariant manifold $M_{\geq \chi_0}$ becomes normally repelling and hence blow-up associated with $M_{\geq \chi_0}$ never occurs, which is consistent with observations in \cite{FV2003}.
One of the original aims in \cite{FV2003} was to construct self-similar profiles $U$ mentioned in (\ref{backward-ss}).
Another parameter $\alpha$ should have the constraint for the backward self-similarity:
$2\beta = (1-m)\alpha - 1$,
which can be determined once $\beta$ is fixed, and $\alpha$ is not explicitly appeared in characterization of blow-up profiles in Theorem \ref{thm-FV}.
\end{rem}

\begin{rem}
Calculations of the Jacobian matrices imply that another branch of equilibria on the horizon; $\{(\chi, 0,0) \mid \chi\in I\}$ for any compact interval $I\subset \mathbb{R}$, is {\em not normally hyperbolic}.
Indeed, this branch contains a center direction in the normal component.
In this case, the center manifold reduction in the corresponding direction will be required to determine the precise blow-up behavior (cf. \cite{I2023_1, IMS2020, IS2020, Mat2019}).
\end{rem}

\subsection{Radially symmetric solutions in a MEMS-like system}
\label{section-MEMS}
The final example is the following system:
\begin{equation}
\label{MEMS-radial}
u'' + \frac{n-1}{r}u' - r^qu^{-p} = 0,\quad ' = \frac{d}{dr},
\end{equation}
where $n\in \mathbb{N}$, $p\in \mathbb{N}$, $q\in \mathbb{R}$ and $r \in (0,\infty)$.
\begin{rem}
\label{rem-IS-MEMS}
The equation (\ref{MEMS-radial}) originates from the following PDE, which is known as a {\em MEMS (micro-electro mechanical system){\color{black}-like system}\footnote{
The present system is different from the original MEMS model.
See e.g. \cite{KS2018} and references therein about the original model and its derivation.
}}:
\begin{equation}
U_t = \Delta U - |x|^q U^{-p},\quad t>0,\quad x\in \mathbb{R}^n.
\end{equation}
Restricting our attention to {\em radially symmetric stationary solutions} $U(t,x) = u(r)$ with $r = |x|$, then $u$ must satisfy (\ref{MEMS-radial}).
In \cite{IS2021}, radially symmetric stationary solutions $u$ with compact supports and derivatives diverging as $r\to r_+-0$ and/or $r\to r_- + 0$ for some $0 < r_- < r_+ < \infty$, namely {\em quenching} solutions, are constructed. 
In particular, all solutions (with $p\in 2\mathbb{N}$) $u(r)$ are negative on these supports.
This work reflects our selection of the directional embedding in (\ref{dir-cpt-MEMS}) below.
\end{rem}
Keeping assertions in Remark \ref{rem-IS-MEMS} in mind, we pay attention to $w = u^{-1}$ instead of $u$ itself:
\begin{equation}
\label{MEMS-radial-2}
2w^{-1}(w')^2 - w''  - \frac{n-1}{r}w' - r^qw^{p+2} = 0,\quad ' = \frac{d}{dr},
\end{equation}
equivalently
\begin{equation}
\label{MEMS-radial-3}
\frac{d}{d\rho}\begin{pmatrix}
r \\ w \\ v
\end{pmatrix} = \begin{pmatrix}
1 \\ v \\ -\frac{n-1}{r}v - r^qw^{p+2} + 2w^{-1}v^2
\end{pmatrix}.
\end{equation}



\begin{lem}
The system (\ref{MEMS-radial-3}) is asymptotically quasi-homogeneous of type $(0,2,p+3)$ and order $p+2$ when $p$ is even, while type $(0,1, (p+3)/2)$ and order $(p+3)/2$ when $p$ is odd.
\end{lem}
In the case of $p$ being even, introducing the directional embedding 
\begin{equation}
\label{dir-cpt-MEMS}
r = r,\quad w = -\frac{1}{s^2},\quad v = \frac{x}{s^{p+3}}
\end{equation}
and the time-scale desingularization
\begin{equation*}
\frac{d\tau}{d\rho} = s^{-(p+1)},
\end{equation*}
we have the following desingularized vector field:
%
%
%
%
%
%

\begin{equation}
\label{MEMS-desing}
\frac{d}{d\tau}\begin{pmatrix}
r \\ s \\ x
\end{pmatrix} = \begin{pmatrix}
s^{p+1} \\ \frac{1}{2}sx \\  \frac{p-1}{2} x^2 -\frac{n-1}{r}s^{p+1} x - r^q
\end{pmatrix}.
\end{equation}
Equilibria on the horizon $\mathcal{E} = \{s=0\}$ exist because $p$ is even\footnote{
When $p$ is odd, the equation requiring  the existence of equilibria on $\mathcal{E}$ is essentially the same as the one associated with (\ref{MEMS-desing}), which yields that there are {\em no} equilibria on $\mathcal{E}$, because the sign of $r^q$ in (\ref{MEMS-desing}) becomes \lq\lq $+$".
}, in which case these are
\begin{equation*}
p_{\pm}(r) = \left( r, 0, \pm \frac{r^{q/2}}{\sqrt{\frac{p-1}{2}}} \right),
\end{equation*}
parameterized by $r > 0$.
The Jacobian matrix on the horizon with $r > 0$ at $(r,0,x)$ is 
\begin{equation*}
\begin{pmatrix}
0 & 0 & 0 \\
0 & \frac{1}{2}x & 0 \\
-qr^{q-1} & 0 & {\color{black}2x(p-1)}
\end{pmatrix},
\end{equation*}
implying that, for any fixed $r_0 > 0$, the invariant set $M_+ = \{p_+(r) \mid r\geq r_0\}$ is normally repelling, while $M_- = \{p_-(r) \mid r\geq r_0\}$ is normally attracting.
In particular, $M_-$ satisfies all requirements in Corollary \ref{cor-nonaut-blowup}.
As a consequence, we have the following theorem.

\begin{thm}
Consider (\ref{MEMS-radial-2}) with $p\in 2\mathbb{N}$.
Then the solution $(r, w(r),v(r))$ through the initial point $(r_0, w_0, v_0)$ with $r_0 > 0$ whose image under the directional embedding (\ref{dir-cpt-MEMS}) is on the (local) stable manifold of $M_{\geq r_0}$ for (\ref{MEMS-desing}) is a blow-up solution with the blow-up rate 
\begin{equation*}
w(r) \sim -c_w(r_{\max} - r)^{-2/(p+1)},\quad v(r) \sim -c_v(r_{\max} - r)^{-(p+3)/(p+1)}\quad \text{ as }\quad r\to r_{\max}-0
\end{equation*}
for some constants $r_{\max} > r_0$ and $c_w, c_v > 0$.
The corresponding solution $u(r)$ of (\ref{MEMS-radial}) has the following asymptotic behavior:
\begin{equation*}
u(r) \sim -c_u(r_{\max} - r)^{2/(p+1)},\quad u'(r) \sim \tilde c_u(r_{\max} - r)^{-(p-1)/(p+1)}\quad \text{ as }\quad r\to r_{\max}-0
\end{equation*}
for some constants $c_u, \tilde c_u > 0$.
In particular, $u(r)$ is a {\em quenching} solution of (\ref{MEMS-radial}).
\end{thm}

The family of blow-up solutions we have obtained is different from those discussed in \cite{IS2021}.
This can be because dominant terms of the vector field at infinity we pay attention to, namely the basic scalings, are different.

\section*{Conclusion}

In this paper, we have developed methodologies to characterize blow-up solutions shadowing \lq\lq trajectories at infinity", and those in nonautonomous systems.
A {\em unified} description of blow-up solutions is organized by means of shadowing trajectories, or geometrically {\em asymptotic phase} of invariant manifolds, and specific exponential decay rates, which generalizes methodologies developed in preceding works \cite{Mat2018, Mat2019}.
Indeed, as investigated in examples (Section \ref{section-blowup-ex}), various finite-time singularities involving blow-ups are extracted by a common machinery.
We strongly believe that the present investigation will be the basis on characterizing blow-up solutions from the viewpoint of dynamical systems.
We end this paper by providing several directions as future works.

\subsection*{Asymptotic expansions and applications to computer-assisted proofs}
One direction of future works would be {\em multi-order asymptotic expansions} of blow-ups, as developed for stationary blow-ups, namely the case where blow-ups are described by equilibria on the horizon when the original vector field is autonomous \cite{asym1, asym2}.
A natural question in this theme towards the generalization is the development of the similar methodology to blow-ups induced by e.g., periodic orbits, NHIMs or others such as chaotic invariant sets on the horizon, and in nonautonomous systems.
In \cite{asym1, asym2}, not only a systematic methodology to calculate asymptotic expansions of blow-up solutions in arbitrary orders, but also a natural correspondence between the essential information characterizing asymptotic expansions and linearized information around equilibria on the horizon are shown.
The latter correspondence enables us to extract blow-up structure in a simple way.
Indeed, as seen in Section \ref{section-Pain}, globally defined compactifications or embeddings require tedious calculations towards the conclusion, which affect the success of applications, such as {\em computer-assisted proofs}, equivalently {\em rigorous numerics} (e.g., \cite{LMT2023, MT2020_1, MT2020_2, TMSTMO2017}).
In general, global embeddings drastically raise the complexity of vector fields, such as types and orders of polynomial, or general asymptotically quasi-homogeneous vector fields, and require quite large computational costs, even for typical numerical computations of blow-ups.
The achievement of the correspondence to asymptotic expansions can reduce the computational costs for blow-ups, in particular for them induced by periodic orbits and NHIMs on the horizon and for nonautonomous systems once the corresponding derivation for these setting is achieved, because the systems involving asymptotic expansions are much simpler than those involving embeddings.
Because applications involving computer-assisted proofs of periodic orbits and NHIMs themselves are widely achieved (e.g., \cite{HCFLM2016}), combination with the present and future machineries will immediately contribute to develop computer-assisted proofs of blow-ups with the above asymptotic behavior at infinity and their {\em global extensions} established in \cite{LMT2023} for stationary blow-ups.

\subsection*{Blow-up rates, spectral intervals and unique ergodicity}
Another natural direction in studies of blow-up solutions is to determine {\em blow-up rates}, the principal terms of blow-up solutions, where the leading term is our interest unlike the multi-order asymptotic expansions.
This is one of central issues in blow-up studies for decades, in particular for PDEs, while all blow-ups mentioned in the present paper are referred to as {\em type-I}, namely the case where the blow-up rates are automatically and uniquely determined by type and order of asymptotically quasi-homogeneous vector fields.
Note that examples exhibiting rates other than type-I are collected in \cite{Mat2019}, where {\em nonhyperbolicity} of invariant sets on the horizon plays a central role in determining special blow-up rates, including ones referred to as {\em type-II} in the field of PDEs.
Although investigations of such blow-up rates with complex asymptotic behavior and in nonautonomous systems are ones of natural direction of future investigations (cf. \cite{K2014}), we would like to leave another direction involving blow-up rates.
\par
One essential condition to determine blow-up rates in the present arguments is the {\em precise exponential rate of convergence to trajectories on the horizon}.
It will be shown in Appendix \ref{section-SSspec} that {\em discrete distribution of dichotomy spectra of linear skew-product flows} provides the required convergence rate in the case of NHIMs.
In special cases mentioned in Proposition \ref{prop-spec-bundle}, dichotomy spectra (or specially eigenvalues and Floquet exponents for constant and periodic cases, respectively) determine the unique exponential decay rates.
It would be also true from comments in \cite{SS1980}, that {\em unique ergodicity} of flows on the base spaces; the hulls of trajectories on NHIMs being minimal sets, could yield the same conclusion.
On the other hand, there are several examples that dichotomy spectra include {\em continuous spectral intervals}, namely intervals with positive Lebesgue measure in the spectra, which are observed when the flow on the base space is non-uniquely ergodic.
Our question here is the following:
\begin{itemize}
\item {\em Is the blow-up rate uniquely determined even when the associated base sets in NHIMs admit more than one ergodic measures, or (\ref{decay-precise-NHIM}) is violated ?}
\end{itemize}
If uniquely determined, its determination with the mechanism is a central issue.
Otherwise, the concrete description of blow-up rate in such a case would be the issue.





\section*{Acknowledgements}
KM was partially supported by World Premier International Research Center Initiative (WPI), Ministry of Education, Culture, Sports, Science and Technology (MEXT), Japan, and JSPS Grant-in-Aid for Scientists (B) (No. JP23K20813).
He would like also to thank Professors Jordi-Llu\'{i}s Figueras and Akitoshi Takayasu for giving him inspiring suggestions to the present study.
{\color{black}He would like finally to the reviewers to provide him with many valuable suggestions to improve the readability of the present paper and simplification of arguments as derivation of fundamental description of blow-ups by means of hyperbolic dynamics at infinity.}

\bibliographystyle{plain}
\bibliography{blow_up_NHIM}


\appendix
\section{Directional embeddings}
\label{sec:dir}

Here a locally defined embedding is introduced, which shall be called a {\em directional embedding}.
\begin{dfn}[Directional embedding, cf. \cite{DLA2006, Mat2018}]\rm
A {\em directional embedding} of type $\alpha = (\alpha_1, \ldots, \alpha_n)$ is defined as 
\begin{align}
\notag
&y = (y_1,\ldots, y_n) \mapsto 
T_{d;\alpha}(y) = (s,\hat x) \equiv (s, \hat x_1,\ldots, \hat x_{i_0 -1}, \hat x_{i_0+1}, \ldots, \hat x_n),\\
\label{dir-cpt}
&y_i := \frac{\hat x_i}{s^{\alpha_i}}\quad (i\not = i_0),\quad y_{i_0} := \pm \frac{1}{s^{\alpha_{i_0}}}
\end{align}
with given direction $i_0$ and the signature $\pm$, provided $i_0\in I_\alpha$.
This embedding is bijective in $U\cap \{\pm y_{i_0}  > 0\}$, in which sense directional embeddings are {\em local} ones.
In particular, this embedding is available when we are interested in solutions of (\ref{ODE-original}) such that the $i_0$-th component has the identical sign during time evolution.
The image of $T_{d;\alpha}$ is
\begin{equation}
\label{D-directional}
\mathcal{D} = \{(s, \hat x_1,\ldots, \hat x_{i_0-1}, \hat x_{i_0+1}, \ldots, \hat x_n) \mid s >0,\, \hat x_i \in \mathbb{R}\, (i \in I_\alpha\setminus\{i_0\}),\, (\hat x_{j_1}, \ldots, \hat x_{j_{n-l}})\in \tilde U\},
\end{equation}
where $\{j_1, \cdots, j_{n-l}\} = \{1,\cdots, n\}\setminus I_\alpha$.
The set $\mathcal{E} = \{s=0\}$ corresponds to the infinity in the original coordinate, which shall be called {\em the horizon}.
\end{dfn}

For simplicity, fix $i_0=1$ in (\ref{dir-cpt}) in the following arguments.
Next transform (\ref{ODE-original}) via (\ref{dir-cpt}), which is straightforward:
\begin{align*}
\frac{ds}{dt} &= -\frac{1}{\alpha_1}s^{-k+1}\hat f_1(s, \hat x_2, \ldots, \hat x_n),\\
\frac{d\hat x_i}{dt} &= s^{-k} \left\{ \hat f_i(s, \hat x_2, \ldots, \hat x_n) -\frac{\alpha_i}{\alpha_1}x_i \hat f_1(s, \hat x_2, \ldots, \hat x_n) \right\}\quad (i=2,\ldots, n),
\end{align*}
where
\begin{equation}
\label{f-tilde-directional}
\hat f_i(s, \hat x_2, \ldots, \hat x_n) \equiv s^{k+\alpha_i} f_i(s^{-\alpha_1}, s^{-\alpha_2}\hat x_2, \ldots, s^{-\alpha_n}\hat x_n),\quad i=1,\ldots, n.
\end{equation}
The resulting vector field is still singular near the horizon, but it turns out that the order of divergence of vector field as $s\to +0$ is $O(s^{-k})$, and hence the following time-scale transformation is available.

\begin{dfn}[Time-variable desingularization: the directional version]\rm 
Define the new time variable $\tau_d$ by
\begin{equation}
\label{time-desing-directional}
d\tau_d = s(t)^{-k} dt
\end{equation}
equivalently,
\begin{equation}
\label{time-desing-directional-integral}
t =  t_0 + \int_{{\color{black}0}}^\tau s(\tau_d)^k d\tau_d,
\end{equation}
where {\color{black}$t_0$ denotes the initial time in the $t$-timescale}, and $s(\tau_d)$ is the solution trajectory $s(t)$ under the parameter $\tau_d$.
We shall call (\ref{time-desing-directional}) {\em the time-variable desingularization (of order $k+1$)}.
\end{dfn}

The vector field $g=g_d$ in $\tau_d$-time-scale is\footnote{
The existence of $B$ follows by cyclic permutations and the fact that $\alpha_1 > 0$. 
See also \cite{Mat2018}.
} 
\begin{equation}
\label{ODE-desing-directional}
\begin{pmatrix}
\frac{ds}{d\tau_d} \\ \frac{dx_2}{d\tau_d} \\ \vdots \\ \frac{dx_{n}}{d\tau_d}
\end{pmatrix}
=
g_d(s,\hat x_2,\ldots, \hat x_n)
\equiv 
\begin{pmatrix}
-s & 0 & \cdots & 0 \\
0 & 1 & \cdots & 0\\
\vdots & \vdots & \ddots & \vdots \\
0 & 0 & \cdots & 1
\end{pmatrix}
B
\begin{pmatrix}
\hat f_1 \\ \hat f_2 \\ \vdots \\ \hat f_n
\end{pmatrix},\quad B = \begin{pmatrix}
\alpha_1 & 0  & \cdots & 0& 0\\
\alpha_2 \hat x_2 & 1 & \cdots & 0 & 0 \\
\vdots & \vdots & \ddots & \vdots & \vdots \\
\alpha_{n-1} \hat x_{n-1} & 0 & \cdots & 1 & 0\\
\alpha_n \hat x_n & 0  & \cdots & 0 & 1
\end{pmatrix}.
\end{equation}
The componentwise expression is
\begin{align*}
\frac{ds}{d\tau_d} &= g_{d,1}(s,\hat x_2,\ldots, \hat x_n) \equiv -\frac{1}{\alpha_1}s \hat f_1(s, \hat x_2, \ldots, \hat x_n),\\
\frac{d\hat x_i}{d\tau_d} &= g_{d,i}(s,\hat x_2,\ldots, \hat x_n) \equiv \hat f_i(s, \hat x_2, \ldots, \hat x_n) -\frac{\alpha_i}{\alpha_1}x_i \hat f_1(s, \hat x_2, \ldots, \hat x_n)\quad (i=2,\ldots, n).
\end{align*}
Note that the above derivation makes sense even if $I_\alpha \not = \{1,\ldots, n\}$ provided $1\in I_\alpha$, namely $\alpha_1 > 0$.
The resulting vector field is as smooth as $f$ {\em including $s=0$} and hence {\em dynamics at infinity} makes sense through dynamics generated by (\ref{ODE-desing-directional}) around the horizon $\mathcal{E} = \{s=0\}$.
Finally note that a correspondence between global/directional embeddings is provided in \cite{Mat2018}.

\section{Discussion 1: Blow-up rates and spectral properties of invariant manifolds}
\label{section-SSspec}

Here we discuss the precise nature of blow-up behavior stated in Theorem \ref{thm-blowup-fund} by means of {\em spectral theory for linear differential systems} (e.g., \cite{SS1974, SS1978, DVV2002_survey, DVV2002, DVV2007}).
\par
\bigskip
An essence to describe (type-I) blow-up rates is the exponential decay with the uniquely determined rate (\ref{decay-precise-NHIM}).
One of the next issue will be {\em when (\ref{decay-precise-NHIM}) is achieved}.
Because the usual decay estimates in characterizing normal hyperbolicity as shown in Proposition \ref{prop-summary-NHIM}, or general hyperbolicity describing asymptotic phase {\color{black}(e.g., \cite{LP2018, LP2021})}, provide only upper bounds of decay rates, which will not be sufficient to obtain (\ref{decay-precise-NHIM}) in general.
In particular, our interests are reduced to the asymptotic behavior in systems of the form
\begin{equation}
\label{linear-differential}
{\bf v}' = A(t){\bf v},
\end{equation}
where $A: \mathbb{R}\to \mathbb{R}^{n\times n}$ denotes a bounded, continuous matrix-valued function.
The spectral theory for (\ref{linear-differential}) is back to Sacker and Sell (e.g., \cite{SS1974}).
Recall that the {\em linear skew-product flow} $\theta$ on the vector bundle $\xi = (\pi, E, X)$ is the pair of continuous maps $\theta = (\varphi, \Phi)$ given by $\theta(t,{\bf v}, {\bf x}) = (\varphi(t,{\bf x}), \Phi(t,{\bf x}){\bf v})$, where $\varphi$ is a flow on the base space $X$ of $\xi$, and $\Phi(t,{\bf x})$ is linear on the fiber $F_{\varphi(t,{\bf x})}$.

\subsection{Review in linear differential systems}
Let $\theta = (\varphi, \Phi)$ be a linear skew-product flow on a vector bundle $\xi = (\pi, E, X)$ with a compact Hausdorff base space $X$.
For each $\lambda \in \mathbb{R}$, define the linear skew-product flow $\theta_\lambda$ by 
\begin{equation*}
\theta_\lambda(t,{\bf v},{\bf x}) \equiv {\color{black}(\Phi_\lambda(t,{\bf x}){\bf v}, \varphi(t,{\bf x})) := (e^{-\lambda t}\Phi(t,{\bf x}){\bf v}, \varphi(t,{\bf x}))}.
\end{equation*}
Stability of solutions for linear skew-product flows is {\color{black}described} in terms of {\em dichotomy} given below, which is a generalization of hyperbolicity for equilibria and periodic orbits.

\begin{dfn}[Exponential dichotomy. e.g., \cite{SS1978}]\rm
Let $M$ be a subset of the base space $X$.
We say that the linear skew-product flow $\theta_\lambda$ admits an {\em exponential dichotomy over $M$} if there is a projector $P: E(M)\to E(M)$ and positive constants $K, a >0$ such that
\begin{align*}
\left|\Phi_\lambda(t,{\bf x}) P({\bf x}) \Phi_\lambda^{-1}(s,{\bf x}) \right| &\leq Ke^{-a(t-s)},\quad s\leq t,\\
\left|\Phi_\lambda(t,{\bf x}) [I - P({\bf x})] \Phi_\lambda^{-1}(s,{\bf x}) \right| &\leq Ke^{-a(s-t)},\quad t\leq s,
\end{align*}
for all ${\bf x}\in M$.
\end{dfn}

The dichotomy determines the {\em spectrum} for linear skew-product flows on vector bundles.
Let $(\xi, \theta)$ be the pair of a vector bundle $\xi$ and a linear skew-product flow $\theta$ on $\xi$.
 
\begin{dfn}[Dichotomy spectrum, cf. \cite{SS1974, SS1978}]\rm
\label{dfn-spec}
For a subset $M\subset X$, define
\begin{align*}
\rho({\bf x}) &:= \{\lambda \in \mathbb{R}\mid \theta_\lambda \text{ has the exponential dichotomy}\}\quad ({\bf x}\in X),\\
\rho(M) &:= \bigcap_{{\bf x}\in M}\rho({\bf x}),\quad 
\Sigma(M) := \mathbb{R}\setminus \rho(M).
\end{align*}
We call $\Sigma(M)$ the {\em dichotomy spectrum} of $(\xi, \theta)$ over $M$. 
\end{dfn}

One of benefits to apply dichotomy spectra is that {\em decomposition of solutions to (\ref{linear-differential}) into invariant subbundles} is applied.
It is known that the dichotomy spectrum $\Sigma(M)$ for a compact invariant set $M$ consists of at most $n$ mutually disjoint compact intervals, which is partially characterized as follows.
 
\begin{lem}[cf. \cite{SS1978}, Lemma 6]
\label{lem-bounds-spec}
Under the setting in Definition \ref{dfn-spec}, let $M$ be a compact invariant set {\color{black}for $\varphi$} in $X$.
Then the dichotomy spectrum $\Sigma(M)$ is compact. 
More precisely, we have
\begin{equation}
\label{dic-spec-M}
\Sigma(M) = [a_1, b_1]\cup \cdots \cup [a_k, b_k]
\end{equation}
with\footnote{
The notation in the present subsection of $k$ is independent of the order of asymptotically quasi-homogeneous vector fields.
} $k \leq n = \dim E(M)$ and
$a_1 \leq b_1 < a_2 \leq  \cdots \leq b_{k-1} < a_k \leq b_k$.
\end{lem}

\begin{prop}[Invariant subbundles, e.g., \cite{SS1978}]
\label{prop-spec-bundle}
Let $M$ be a compact invariant set in $X$ and $\Sigma(M)$ be the dichotomy spectrum of $(\xi|_M, \theta)$ given by (\ref{dic-spec-M}).
Then the vector bundle $E(M)$ over $M$ is decomposed into
\begin{equation}
\label{inv_subbundle}
E(M) = E_1\oplus \cdots \oplus E_k
\end{equation}
as the Whitney sum, such that each subbundle $E_i$ over $M$ is associated with the spectral interval $[a_i, b_i]$ and invariant under $\theta$.
When $\lambda \in (b_i, a_{i+1})$ for some $i\in \{1, \ldots, k-1\}$, we have
\begin{align*}
&\lim_{t\to +\infty}\left\|e^{-\lambda t} \Phi(t,{\bf x}){\bf v} \right\|\to 0\quad \text{ if }\quad ({\bf v},{\bf x})\in E_1\oplus \cdots \oplus E_i,\\
&\lim_{t\to -\infty}\left\|e^{-\lambda t} \Phi(t,{\bf x}){\bf v} \right\|\to 0\quad \text{ if }\quad ({\bf v},{\bf x})\in E_{i+1}\oplus \cdots \oplus E_k.
\end{align*}

\end{prop}

Once we detect the location of trajectories for the linear skew-product flow on $(\xi|_M, \theta)$, its spectral property is extracted by means of invariant spectral subbundles, which can vary in time-evolution within the specified spectral intervals.
The simplest case will be the {\em discrete} distribution of spectrum.

\begin{prop}[Discrete distribution of spectrum. e.g., \cite{SS1978}]
\label{prop-discrete-spec}
For a pair $(\xi, \theta)$ mentioned with $\dim \xi = n$, let $X$ be the compact Hausdorff base space of $\xi$ and $\varphi$ be the flow on $X$ determining $\theta$.
If one of the following for the trajectory ${\bf x}(t)$ on $X$ for $\varphi$ holds true, then $\Sigma(M)$ consists of finite points, where $M = \overline{\{{\bf x}(t)\}}\equiv H({\bf x}(t))$ denotes the {\em hull} of ${\bf x}(t)$;
\begin{itemize}
\item ${\bf x}(t)$ is an equilibrium;
\item ${\bf x}(t)$ is a time-periodic solution;
\item ${\bf x}(t)$ is almost periodic (cf. \cite{S1971, F1974}) and the spectrum $\Sigma(M)$ consists of $n$-mutually disjoint closed intervals
\end{itemize}
\end{prop}

It is also mentioned in \cite{SS1978} that this claim is also true for linear skew-product flows on the vector bundle $\xi$ over the base space $X$ being compact, {\em minimal} and a {\em uniquely ergodic measure} $\mu$ ({\em unique ergodicity} of sets).
Indeed, \cite{SS1978} indicates that it is possible to have a nontrivial dichotomy spectral interval when one does not have unique ergodicity.
We shall refer to e.g., \cite{O1952, S1971} for ergodicity, because we will not discuss ergodicity of dynamical systems in the present paper, while it is left to future works.
See e.g., \cite{S1971} for detailed arguments of minimal sets.

\begin{rem}[Lyapunov exponents and dichotomy spectrum]
{\em Lyapunov exponents} are another well-known quantities to characterize exponential behavior of solutions to (\ref{linear-differential}).
See e.g., \cite{A1995} for details.
Lyapunov exponents determine the {\em Lyapunov spectrum}, the collection of intervals between upper and lower Lyapunov exponents.
It is well known that Lyapunov spectrum is generally a proper subset of dichotomy spectrum, indicating that Lyapunov exponents are {\em not} suitable for characterizing decaying/expanding properties of solutions to (\ref{linear-differential}).
See e.g., \cite{DVV2002_survey, DVV2002, DVV2007}.
On the other hand, it is shown in e.g., \cite{SS1978} that, if a spectral interval $[a_i, b_i]$ in the dichotomy spectrum consists of one point; $a_i = b_i$, then the quantity coincides with the corresponding Lyapunov exponent, namely the exponent included in $[a_i, b_i]$.
\end{rem}

\subsection{Smooth conjugacy near compact invariant manifolds}
\label{section-conjugacy}
One classical property in the theory of NHIMs not mentioned in Section \ref{section-preliminaries-NH} is {\em local conjugacy of dynamics to linearizations}.
Numerous works involving this property have beed developed; e.g., \cite{F1971, HPS1977, R1971, S1994, S1983, S1984} as generalizations of {\em Hartman-Grobman's Theorem} stating the linearization property around hyperbolic equilibria.
In general, only $C^0$-smoothness is permitted for local conjugacy around NHIMs (e.g., \cite{F1971, HPS1977}), and additional properties involving hyperbolicity are required to ensure the $C^r$-smoothness of conjugacy with $r\geq 1$ (e.g., \cite{F1971, HPS1977, S1994}).
Here we refer to arguments in \cite{S1984} and related works \cite{R1971, S1983}.
\par
\bigskip
For a given smooth vector field
\begin{equation}
\label{ODE-aut-general}
{\bf y}' = f({\bf y})
\end{equation}
on $\mathbb{R}^n$, assume that it admits a compact, boundaryless $(\partial M = \emptyset)$, smooth NHIM $M \subset E$.
Without the loss of generality, we assume that the normal bundle
\begin{equation*}
NM \equiv E^s\oplus E^u,\quad \text{ where } \quad T_M\mathbb{R}^n = TM\oplus E^s\oplus E^u
\end{equation*}
over $M$ is {\em trivial}, namely $NM = M\times F$ for a fiber $F\cong \mathbb{R}^d$.
 
\begin{rem}[Trivialization of normal bundles and associated dynamics. \cite{S1994}]
If the bundle $NM$ is not trivial, we can construct a following extended trivial bundle and associated vector field. 
In general, any vector bundles admit a complementary vector spaces so that the bundles with the extended fibers are trivial (e.g., \cite{MS1974}).
If $d_1$ and $d_2$ are the dimensions of trivializing bundles of $E^s$ and $E^u$, namely
\begin{equation*}
E^s\oplus (M\times \mathbb{R}^{d_1}) = M\times \mathbb{R}^{d_1 + n_s},\quad E^u\oplus (M\times \mathbb{R}^{d_2}) = M\times \mathbb{R}^{d_2 + n_u},
\end{equation*}
then the system (\ref{ODE-aut-general}) can be extended by adding extra equations $\dot {\bf v} = -\lambda {\bf v}$ and $\dot {\bf w} = \lambda {\bf w}$, where ${\bf v}\in \mathbb{R}^{d_1}$, ${\bf w}\in \mathbb{R}^{d_2}$ and $\lambda = (a + b)/2$, where $a, b \geq 0$ are characterized by (\ref{NHIM-spec}) below.
This modification does not affect the conditions on (\ref{ODE-aut-general}).
\end{rem}

In the above case, we can introduce curvilinear local coordinates so that, in a small neighborhood of $M$, (\ref{ODE-aut-general}) has the following expression:
\begin{equation}
\label{ODE-on-normal}
{\bf x}' = A(\theta){\bf x} + F({\bf x}, \theta),\quad \theta' = g(\theta) + G({\bf x}, \theta),
\end{equation}
where ${\bf x}\in F$ represents the normal vector to $M$ (i.e., elements on the fiber) and $\theta$ represents local coordinates on $M$.
Here the invariant manifold $M$ is identified with the $0$-section of the normal bundle $NM$, namely ${\bf x}\equiv 0$ corresponds to the property $\theta(t) \in M$ for all $t\in \mathbb{R}$ and $\theta_0 \equiv \theta(0)\in M$.

\par
\bigskip
Following \cite{S1983}, introduce a stronger requirement on spectra.
\begin{dfn}[Normal and tangential spectra, e.g., \cite{SS1978, S1983, S1984}]\rm
For a system of the form (\ref{ODE-on-normal}), let $\Sigma_N$ and $\Sigma_T$ be the (dichotomy) spectra of the linear skew-product flows generated by
\begin{align*}
{\bf x}' &= (A(\phi) - \lambda I){\bf x},\quad \phi' = g(\phi),\\
{\bf y}' &= (Dg(\phi) - \lambda I){\bf y},\quad \phi' = g(\phi),
\end{align*}
respectively.
These sets are called {\em normal spectrum} and {\em tangential spectrum}.
\end{dfn}
For the system (\ref{ODE-on-normal}), the manifold $M$ being an $r$-NHIM is characterized by $(0 \leq) ra < b$, where
\begin{equation}
\label{NHIM-spec}
\lambda \in \Sigma_T \Rightarrow |\lambda| \leq a,\quad \lambda \in \Sigma_N \Rightarrow |\lambda| \geq b.
\end{equation}

\begin{dfn}[Sternberg-Sell conditions, \cite{S1984}]\rm
\label{dfn-Sternberg-Sell}
Let $\Sigma_N$ and $\Sigma_T$ be the normal and tangential spectra associated with (\ref{ODE-on-normal}), and let $a$ and be as in (\ref{NHIM-spec}).
Assume that $\Sigma_N = I_1\cup \cdots \cup I_l$, each compact interval $I_i$ associates the spectral subbundle $V_i$ with $n_i = \dim V_i$.
We say that the $d$-tuple of real numbers $(\lambda_1, \ldots, \lambda_l)$ is {\em admissible} if 
\begin{enumerate}
\item the mapping $j\mapsto \lambda_j$ from $\{1,\ldots, d\}$ to $\mathbb{R}$ has the range in $\Sigma_N$, and 
\item $\sharp \{j \mid \lambda_j \in [a_i, b_i]\} = n_i$, $1\leq i \leq l$.
\end{enumerate}
The matrix $A$ is said to satisfy the {\em Sternberg-Sell condition of order $N$} if 
\begin{equation}
\label{Sternberg-Sell}
\begin{cases}
|\lambda - (m_1 \lambda_1 + \cdots + m_d \lambda_d)| > Na & \\
| m_1 \lambda_1 + \cdots + m_d \lambda_d | > (N+1)a & \\
\end{cases}
\end{equation}
holds for all $\lambda \in \Sigma_N$ and all admissible $d$-tuples $(\lambda_1, \ldots, \lambda_d)$ and nonnegative integers $m_1, \ldots, m_d$ satisfying $2\leq m_1 + \cdots + m_d \leq 2N$.
\end{dfn}

\begin{prop}[Smooth conjugacy, \cite{S1983, S1984}]
Consider (\ref{ODE-on-normal}) near a $r$-NHIM $M$ where the coefficients are $C^{4r}$.
Let $a$ and $b$ as (\ref{NHIM-spec}).
If the Sternberg-Sell condition of order $r$ holds, then there is a $C^r$-conjugacy between (\ref{ODE-on-normal}) and the linearization
\begin{equation*}
{\bf y}' = A(\phi){\bf y},\quad \phi' = g(\phi).
\end{equation*}
\end{prop}
We shall use this result with $r=1$ in the following arguments as well as applications, e.g., Section \ref{section-blowup-ex}.
\begin{rem}
\label{rem-resonance}
Condition (\ref{Sternberg-Sell}) is a generalization of {\em non-resonance}, which ensures the smooth conjugacy of vector fields around equilibria or periodic orbits (e.g., \cite{CFdlL2005}). 
The difference is that (\ref{Sternberg-Sell}) is required for all $\lambda$ on the dichotomy spectrum which is {\em continuously} distributed in general.
In case that the matrix $A$ is constant, then the spectrum is discrete (cf. Proposition \ref{prop-discrete-spec}) and hence the Sternberg-Sell condition is reduced to well-known non-resonance condition\footnote{
This type of conditions originates from Sternberg (e.g., \cite{St1958}) studying obstructions of the existence of smooth conjugacy of diffeomorphisms (and vector fields) around fixed points.
Later, Sell \cite{S1983, S1984} has developed the generalization to compact invariant manifolds stated here, based on Robinson \cite{R1971}.
}.
\end{rem}

The above arguments will be suitable for a wide range of dynamical systems and applications to e.g., {\em parameterizations of invariant manifolds} as well as their numerical computations and computer-assisted proofs (e.g., \cite{HCFLM2016, LMT2023}).

\subsection{Asymptotic behavior and blow-up rates}

Now apply the above frameworks to the desingularized vector field $g$ given in (\ref{desing-para}) associated with the (autonomous) vector field $f$.
\begin{prop}
\label{prop-req-NHIM-blowup}
Under the setting of Corollary \ref{cor-NH-blowup-1}, assume that the normal bundle $NM$ is trivial.
Consider the system $g$ around $M$ of the form (\ref{ODE-on-normal});
\begin{equation}
\label{desing-para-normal}
\dot {\bf v} = A(\theta){\bf v} + G_M({\bf v}, \theta),\quad \dot \theta = g_M(\theta) + F_M({\bf v}, \theta),\quad \dot{} = \frac{d}{d\tau},
\end{equation}
where ${\bf v}$ denotes the normal vector of $NM$ and $\theta \in M$.
Let $N\varphi_g$ be the linear skew-product flow generated by (\ref{desing-para-normal}).
Assume that (\ref{desing-para-normal}) is $C^4$ around $M$ and Sell-Sternberg condition of order $1$ holds.
%
Further suppose that the dichotomy spectrum $\Sigma_{E^s}(V_{{\bf x}_\gamma})$ of $(E^s(V_{{\bf x}_\gamma}), N\varphi_g)$ restricted to the stable subbundle $E^s(V_{{\bf x}_\gamma})$ over the hull $V_{{\bf x}_\gamma} \equiv H({\bf x}_\gamma(\tau)) \subset M$ is discrete:
\begin{equation*}
\Sigma_{E^s}(V_{{\bf x}_\gamma}) = \{ \lambda_1^s, \ldots, \lambda_{m_s}^s \}\subset (-\infty,0),
\end{equation*}
where $m_s \leq n_s$.
Then the exponential estimate (\ref{decay-precise-NHIM}) holds.
\end{prop}

\begin{proof}
Let ${\bf x}(\tau) = T_{{\rm para};\alpha}({\bf y}(\tau))$ and denote $E^s = E^s(V_{{\bf x}_\gamma})$ for simplicity.
Our assumption implies that the bundle $NM = E^s \oplus E^u$ admits the following Whitney sum decomposition 
into spectral invariant subbundles (Proposition \ref{prop-spec-bundle}):
\begin{equation}
\label{spectral-N}
E^s = E_1^s \oplus \ldots \oplus E_{m_s}^s,
\end{equation}
$E_i^s$ is the spectral invariant subbundle associated with the point spectrum $\{\lambda_i^s\}$, $i=1,\ldots, m_s$.
The Sternberg-Sell condition assures the existence of a $C^1$-conjugacy $h$ between (\ref{desing-para-normal}) and its linearized flow on $E^s$.
Note that the conjugacy $h$ preserves $M$, namely $h({\bf x}) = ({\bf 0}, {\bf x})$ for all ${\bf x}\in M$.
Invariant structure of the spectral decomposition indicates that the fiber component ${\bf v}(\tau)$ of a trajectory $({\bf v}(\tau), {\bf x}_\gamma(\tau)) = h({\bf x}(\tau))$ of $N\varphi_g$ with ${\bf x}(\tau) = T_{{\rm para}; \alpha}({\bf y}(\tau))$ and the topological conjugacy $h$ is located on the subbundle 
\begin{equation}
\label{location-v}
E^s_{i_1}\oplus \cdots \oplus E^s_{i_p}\quad \text{ with }\quad i_1, \ldots, i_p \in \{1,\ldots, m_s\}\quad \text{ and }\quad i_1< \cdots < i_p 
\end{equation}
for all $\tau \in \mathbb{R}$.
In particular, ${\bf v}(\tau)$ is uniquely decomposed into
\begin{align*}
&{\bf v}(\tau) = v_{i_1}(\tau) + \cdots + v_{i_p}(\tau),\quad \text{where} \quad (v_{i_l}(\tau), {\bf x}_\gamma(\tau))\in E_{i_l}^s \quad \text{ with }\quad v_{i_l}(\tau)\not = 0\, \text{ for all }\quad l=1,\ldots, p.
\end{align*}
Without the loss of generality, we may assume that $\lambda_1 < \cdots < \lambda_{m_s} < 0$.
By the characterization of exponential dichotomy, each component $v_{i_l}(\tau)$ satisfies 
\begin{equation*}
e^{ (-\lambda_{i_l}+ \epsilon )\tau} |v_{i_l}(\tau)| = \begin{cases}
+\infty & \epsilon > 0 \\
0 & \epsilon < 0 
\end{cases}\quad \text{ as }\quad \tau\to +\infty.
\end{equation*}
The original solution ${\bf x}(\tau)$ as the image of such a trajectory through $h$ is written by
\begin{align*}
{\bf x}(\tau) &= H({\bf v}(\tau), {\bf x}_\gamma(\tau)) \equiv h^{-1}({\bf v}(\tau), {\bf x}_\gamma(\tau)) \\
	&= H({\bf 0}, {\bf x}_\gamma(\tau)) + D_{}H(\tilde c{\bf v}(\tau), {\bf x}_\gamma(\tau))\{({\bf v}(\tau), {\bf x}_\gamma(\tau)) - ({\bf 0}, {\bf x}_\gamma(\tau))\}\\
	&={\bf x}_\gamma(\tau) + D_{}H(\tilde c{\bf v}(\tau), {\bf x}_\gamma(\tau))({\bf v}(\tau), {\bf 0})
\end{align*}
with $\tilde c = \tilde c(\tau)\in (0,1)$
via the Taylor's theorem.
From the exponential decay of ${\bf v}(\tau)$ and continuity of $DH$ with the mentioned property in a neighborhood of $M$, 
we have
\begin{align*}
e^{(-\lambda_i + \epsilon)\tau}\left| x_i(\tau) - x_{i;\gamma}(\tau) \right| &= e^{(-\lambda_i + \epsilon)\tau}\left| \{ DH({\bf v}(\tau), {\bf x}_\gamma(\tau))({\bf v}(\tau), {\bf 0})\}_i \right| \\
	& = \begin{cases}
+\infty & \epsilon > 0 \\
0 & \epsilon < 0 
\end{cases}\quad \text{ as }\quad \tau\to +\infty
\end{align*}
with some $-\lambda_i \leq \lambda_s < 0$ for all $i=1,\ldots, n$, where $\lambda_s \in \{\lambda_i\}_{i=1}^{m_s}$ is uniformly determined depending on initial points of ${\bf x}(\tau)$.
Because
\begin{align*}
1-p({\bf x}(\tau))^{2c} &= 1 - \sum_{i\in I_\alpha}x_i(\tau)^{2\beta_i} \\
\label{expansion-1mp}
	&= \sum_{i\in I_\alpha} ( x_{\gamma,i}(\tau) - x_i(\tau) )(x_{\gamma,i}(\tau) + x_i(\tau)) \left\{ \sum_{j=0}^{\beta_i - 1} x_{\gamma,i}(\tau)^{2(\beta_i-1-j)}x_i(\tau)^{2j} \right\}\\
	&= \sum_{i\in I_\alpha} \{ -DH(\tilde c_i(\tau){\bf v}(\tau), {\bf x}_\gamma(\tau))({\bf v}(\tau), {\bf 0})\}_i (x_{\gamma,i}(\tau) + x_i(\tau)) \left\{ \sum_{j=0}^{\beta_i - 1} x_{\gamma,i}(\tau)^{2(\beta_i-1-j)}x_i(\tau)^{2j} \right\}
\end{align*}
for functions $\tilde c_i(\tau)\in (0,1)$, we have
\begin{equation*}
e^{(-\lambda_i +\epsilon) \tau}(1-p({\bf x}(\tau))^{2c}) = \begin{cases}
+\infty & \epsilon > 0 \\
0 & \epsilon < 0 
\end{cases}\quad \text{ as }\quad \tau\to +\infty
\end{equation*}
with some $-\lambda_i \leq \lambda_s < 0$ for all $i=1,\ldots, n$, where $\lambda_s \in \{\lambda_i\}_{i=1}^{m_s}$ is uniformly determined.
\end{proof}

As a consequence, for NHIMs admitting the spectral information mentioned in the proposition, Theorem \ref{thm-blowup-fund} can be applied.
In particular, combining Proposition \ref{prop-discrete-spec}, we can rephrase results in the preceding paper \cite{Mat2018}.
\begin{cor}[Stationary and periodic blow-ups, \cite{Mat2018}]
In Theorem \ref{cor-NH-blowup-1}, assume that $M$ is either a hyperbolic equilibrium or a hyperbolic periodic orbit.
{\color{black}We further assume that (\ref{Sternberg-Sell}) with $a=0$ holds.}
Then all statements in Theorem \ref{thm-blowup-fund} hold.
\end{cor}

Blow-ups in {\em backward} time direction can be described in the similar way, in which case trajectories on the local {\em unstable} manifold $W_{\rm loc}^u(M;\varphi_g)$ determine the behavior.

\section{Discussion 2: Blow-up rates in nonautonomous blow-ups}
\label{section-technical_nonaut}

This appendix addresses several practical ideas to apply Theorem \ref{thm-blowup-nonaut-fund} involving blow-up rates in concrete nonautonomous systems, and their validity to prove the concrete blow-up behavior.

\subsection{Blow-up rates}
\label{section-boundaryless-mod}

One approach is to use the idea stated in Proposition \ref{prop-req-NHIM-blowup} in autonomous cases; the spectral information of $M_I$, the tube of an invariant manifold over $I\subset \mathbb{R}$ given in (\ref{M_slice}), for the {\em linearized} systems.
In Proposition \ref{prop-req-NHIM-blowup}, the spectral information is applied through the local smooth conjugacy, which is generally admitted for {\em boundaryless} NHIMs (cf. \cite{HPS1977, PS1970, S1983, S1984}).
Because our invariant manifolds (for the modified vector field $\tilde g$ given in (\ref{desing-para-nonaut-mod})) of the form $M_{I'}$ have nontrivial boundary, some modifications are necessary if we apply the above conjugacy.
Here we refer to a result in \cite{EKR2018}, where global linearization of (normally attracting) invariant manifolds is provided, {\em no matter when invariant manifolds have boundaries}.
The technique relies on several boundaryless treatments in differential topology (e.g., \cite{H1994}) and the {\em uniformity lemma}\footnote{
If we try to provide the similar result to invariant manifolds which are not normally hyperbolic, the corresponding estimate will be necessary.
} stemming from normally hyperbolic structure of manifolds (e.g., \cite{F1971, F1974_rate, W2013}) to ensure persistence of normally hyperbolic structure in perturbed vector fields.

\begin{prop}
\label{prop-embed-MI}
Consider the desingularized vector field $g$ in (\ref{desing-para-nonaut}) satisfying requirements in Assumption \ref{ass-nonaut-inv}.
Then there are a Riemannian manifold $\hat Q$ and a vector field $\hat g$ on $\hat Q$ such that 
\begin{itemize}
\item $M_I$ is embedded into a compact, boundaryless, normally attracting invariant manifold $\hat N \subset \hat Q$, 
\item $M_I$ is normally attracting for $\hat g$ as the submanifold of $Q$, and that
\item $\hat g = g$ on $W^s(M_I; \varphi_{\hat g})$ as the submanifold of $W^s(\hat N;\varphi_{\hat g})$, where $\varphi_{\hat g}$ is the flow generated by $\hat g$.
In particular, the stable foliation of $W^s(M_I; \varphi_{\hat g})$ is constructed, which is consequently regarded as the foliation of $W^s(M_I; \varphi_g)$.
\end{itemize}
\end{prop}

\begin{proof}
Let $\tilde g$ be the modified desingularized vector field given in (\ref{desing-para-nonaut-mod}) admitting a sequence of normally hyperbolic manifolds $M_I \subset M_{\tilde I} \subset M_{I'}$ such that $M_I$ is invariant, while $ M_{\tilde I}$ and $M_{I'}$ are inflowing invariant.
Let $Q := W^s(M_{I'}; \varphi_{\tilde g})$, the global stable manifold of $M_{I'}$ for $\tilde g$.
We then apply Proposition \ref{prop-EKR-boundaryless} below to $N = M_{\tilde I}$ and $M = M_{\hat I}$, which directly yields the result.
\end{proof}


Once the treatment of invariant manifolds on $\mathcal{E}$ admitting normally hyperbolic structure is understood, blow-up rates are evaluated in the similar way to autonomous cases.

\begin{cor}[Nonautonomous blow-up: blow-up rates]
\label{cor-nonaut-blowup-2}
Under the setting in Corollary \ref{cor-nonaut-blowup-1}, if all assumptions in Proposition \ref{prop-req-NHIM-blowup} are satisfied for the hull ${\color{black}V_{{\bf x}_\gamma} }$, then the blow-up solution ${\bf y}(t)$ admits the asymptotic behavior
\begin{align*}
p({\bf y}(t)) \sim C_0(-\ln (t_{\max}-t)) (t_{\max}-t)^{-1/k}\quad \text{ as }\quad t \to t_{\max}-0
\end{align*}
for some function $C_0(-\ln (t_{\max}-t))$ satisfying
\begin{equation*}
C_0(\tau) = o(e^{\epsilon \tau}),\quad C_0(\tau)^{-1} = o(e^{\epsilon \tau})
\end{equation*}
for any $\epsilon > 0$ as $\tau \to \infty$, equivalently $-\ln (t_{\max}-t) \to +\infty$, and
\begin{equation*}
y_i(t) \sim C_0(-\ln (t_{\max}-t))^{\alpha_i} x_{\gamma,i} (-\bar c\ln(t_{\max} - t)) (t_{\max}-t)^{-\alpha_i /k} \quad \text{ as }\quad t \to t_{\max}-0
\end{equation*} 
for some constant $\bar c > 0$, provided $x_{M,i} (\tau)\not \to 0$ as $\tau \to \infty$.
\end{cor}

\begin{proof}
It follows from Proposition \ref{prop-embed-MI} that $M_I$ as the submanifold of $Q = W^s(M_{I'}; \varphi_g)$ admits the stable foliation of $W_{\rm loc}^s(M_I; \varphi_{\hat g}) = W_{\rm loc}^s(M_I; \varphi_g)$.
The statement then follows from the distribution of $M_I\subset \mathcal{E}$, in particular of $\pi_t M_I$ in $\mathbb{R}$.
Results involving asymptotic behavior of $p({\bf y}(t))$ and $y_i(t)$ are the direct consequence of Proposition \ref{prop-req-NHIM-blowup} replacing $M$ with $\hat N$ stated in Proposition \ref{prop-embed-MI}.
\end{proof}

\begin{rem}[Beyond normal hyperbolicity]
Here we have restricted our attention to NHIMs because our geometric description of blow-ups relies on embeddings of compact invariant manifolds with boundary to compact \lq\lq boundaryless" manifolds with the same normal hyperbolicity restricted to the original manifolds.
This technique is based on \cite{EKR2018} and its validity beyond normal hyperbolicity (like situations in \cite{LP2021}) remains nontrivial.
If the similar embedding technique is constructed to general invariant manifolds (admitting asymptotic phase), the same conclusion for blow-up rates will be provided.
\end{rem}

\subsection{Proof of Corollary \ref{cor-nonaut-blowup}}
\label{section-app_proof-nonaut}

The proof is a direct consequence of combinations of Proposition \ref{prop-discrete-spec} and Corollary \ref{cor-nonaut-blowup-2}.
\par
First, thanks to Proposition \ref{prop-embed-MI}, the manifold $M_I$ is embedded into a compact, boundaryless NAIMs in a manifold $\hat Q$ (mentioned in Proposition \ref{prop-EKR-boundaryless} for $Q = W^s(M_{I'}; \varphi_{\tilde g})$), and invariant foliation of $W^s_{\rm loc}(M_I; g)$ is admitted.
In particular, $g$ admits a local topological conjugacy on $M_I$ to $Dg|_{\hat Q}$, where we have used the identification of $g$ and the modified vector field $\hat g$ on $W^s_{\rm loc}(M_I; g)$.
Next, because $M_I$ is assumed to be a collection of equilibria, the dichotomy spectrum of $Dg$ on each hull $V_{{\bf x}_\gamma}$ of trajectories on $M_I$, which is exactly a point ${\bf x}(t)\in M_I$ itself for each $t\in I$, are discrete by Proposition \ref{prop-discrete-spec}.
Therefore, Corollary \ref{cor-nonaut-blowup-2} can be applied to obtaining the conclusion.

\subsection{Reduction to boundaryless inflowing NAIMs}
\label{section-app_bo_bd}

Finally we review a technique to modify inflowing invariant manifolds with normally hyperbolic structure with non-trivial boundary to {\em boundaryless} ones used in the argument of Proposition \ref{prop-embed-MI}.
This was provided in \cite{EKR2018}, where the reduction of inflowing {\em normally attracting} invariant manifolds (NAIMs) with boundaries to boundaryless ones was discussed.
The main aim of the discussion was to apply global linearization theorem for boundaryless NHIMs stated in preceding works (e.g. \cite{HPS1977, PS1970}) to inflowing NAIMs.
See e.g., \cite{H1994} for fundamental concepts in differential geometry and differential topology.
Let $Q$ be a $C^{r\geq 1}$ Riemannian manifold.

\begin{prop}[Reduction to boundaryless NAIM, \cite{EKR2018}, Proposition B.1]
\label{prop-EKR-boundaryless}
Let $M,N\subset Q$ be compact inflowing $r$-NAIMs, with $M\subset {\rm int}_Q N$, for the $C^{r\geq 1}$ flow $\varphi$ generated by the $C^{r\geq 1}$ vector field $f$ on $Q$.
Let $U_0$ be an arbitrarily small tubular neighborhood of $\partial N$, having smooth boundary $\partial U_0$ and disjoint from $W^s_{\rm loc}(M;\varphi)$.
Define $\hat Q$ to be the double of $Q\setminus U_0$.
\par
Then there is a $C^\infty$ differential structure on $\hat Q$ and a $C^r$ vector field $\hat f: \hat Q\to T\hat Q$ satisfying the following properties:
\begin{enumerate}
\item $\hat f$ is equal to $f$ on each copy of $Q\setminus U_0$ except on an arbitrarily small neighborhood of $\partial U_0$.
\item There is a compact and boundaryless $r$-NAIM $\hat N$ for $\hat f$, with $\hat N$ equal to $N$ on each copy of $Q\setminus U_0$, except on an arbitrarily small neighborhood of $\partial U_0$.
\item The global stable foliation of $M$ for $f$ does not intersect $U_0$, and it coincides with the global stable foliation of $M$ for $\hat f$, where $M$ and $W^s(M)$ are identified via inclusion with subset of a copy of $Q\setminus U_0$ in $\hat Q$.
\end{enumerate}
Furthermore, let $\hat \varphi$ be the $C^r$-flow generated by $\hat f$, and let $\hat E^s$ be the $D\hat \varphi^t|_{\hat N}$-invariant stable vector bundle for the NAIM $\hat N$.
If, additionally, there are constants $C > 0$ and $\alpha < 0$ such that, for all $m\in M$, $t\geq 0$ and $0\leq i \leq k$, the $k$-center bunching condition\footnote{
This condition implies that $W^s(M;\varphi)$ and $E^s$ as fiber or vector bundles are $C^k$, according to \cite{F1977}.
}
\begin{equation}
\label{k-bunch}
\|D\varphi^t |_{T_mM}\|^i \|D\varphi^t |_{E^s_m}\| \leq Ce^{\alpha t} m\left(  D\varphi^t |_{T_mM}\right)
\end{equation}
is satisfied for the original system on $Q$, then (\ref{k-bunch}) is also satisfied replacing $M$, $E^s$ and $\varphi^t$ with $\hat N$, $\hat E^s$ and $\hat \varphi^t$, respectively, and $\alpha$ with some different constant $\hat \alpha < 0$.
\par
Similarly, if additionally there are constants $0 < \delta < -\alpha < -\beta$ and $C\geq 1$ such that, for all $t\geq 0$,
\begin{equation}
\label{EKR-dic}
\begin{cases}
C^{-1}e^{-\delta t} \leq m\left( D\varphi^t |_{TM} \right) \leq \| D\varphi^t |_{TM} \| \leq Ce^{\delta t}, & \\
C^{-1}e^{-\delta t} \leq m\left( \left( D\varphi^t |_{TM} \right)^{-1} \right) \leq \| \left( D\varphi^t |_{TM} \right)^{-1} \| \leq Ce^{\delta t}, & \\
C^{-1}e^{\beta t} \leq m\left( D\varphi^t |_{E^s} \right) \leq \| D\varphi^t |_{E^s} \| \leq Ce^{\alpha t} &
\end{cases}
\end{equation}
uniformly on $TM$ and $E^s$, then we can choose $\hat f$ appropriately such that the same conclusion holds true for $\hat \varphi^t$, $T\hat N$ and $\hat E^s$ with modified constants $0 < \hat \delta < -\hat \alpha < -\hat \beta$ arbitrarily close to $\delta, \alpha, \beta$, respectively.
\end{prop}

The essence of this result is that the differential structure of manifolds and vector fields are modified only in a neighborhood of $\partial U_0$, although the ambient manifold $Q$ is implicitly assumed to be boundaryless.
Even when $Q$ has a nontrivial boundary $\partial Q$, the above technique can be applied as long as $(Q, \partial Q)$ is a Riemannian manifold and $\partial Q$ is invariant for the vector field $f: Q\to TQ$. 
In this result, {\em normally attracting}, or {\em normally hyperbolic} properties of $N$ and $M$ is assumed. 
In this situation, the {\em uniformity lemma}\footnote{
If we try to provide the similar result to invariant manifolds which are not normally hyperbolic, the corresponding estimate will be necessary.
} can be applied (e.g., \cite{F1971, F1974_rate, W2013}) to making uniform estimates in (\ref{EKR-dic}).
\par
\bigskip
In Appendix \ref{section-boundaryless-mod}, this proposition is used to provide linearization of dynamics at infinity for nonautonomous systems.
More precisely, this technique is used to characterize spectral properties for invariant manifolds in nonautonomous systems by means of {\em compact, boundaryless} invariant manifolds, where local topological {\color{black}(and smooth)} conjugacy can be applied.
Although our interests are {\em NHIMs}, all arguments involving NAIMs can be applied by restricting our attention to the (global) stable manifold: $Q = W^s(M;\varphi)$ with a compact, connected (inflowing) invariant manifold $M$ with a normally hyperbolic structure, then $M$ is normally {\em attracting} for the flow restricted to $Q$, and hence the theory of {\em NAIMs} can be directly applied\footnote{
This is possible because the stable manifold $W^s(M)$ of an invariant manifold $M$ for a given flow is also invariant.
}.


\end{document}